\documentclass[12pt]{article}
\usepackage{amsmath,amssymb,amsthm,latexsym,amscd,stmaryrd,mathrsfs,longtable,color}
\newtheoremstyle{standard}%
{9pt}%
{9pt}%
{\it}
{}%
{\bfseries}%
{}
{ }%
{#3}%

\newcommand{\db}[1]{(\!({#1})\!)}

\numberwithin{equation}{section}

\newcommand{\N}{{\mathbb N}}
\newcommand{\Z}{{\mathbb Z}}

\newcommand{\C}{{\mathbb C}}

\newcommand{\wi}{i}
\newcommand{\wj}{j}
\newcommand{\wk}{k}

\newcommand{\wpp}{p}
\newcommand{\wq}{q}

\DeclareMathOperator{\Hom}{Hom} 
\DeclareMathOperator{\End}{End}

\DeclareMathOperator{\Res}{Res}

\DeclareMathOperator{\Span}{Span}

\DeclareMathOperator{\wt}{wt}

\newcommand{\ewo}{u}
\newcommand{\ewt}{v}

\newcommand{\glb}[2]{\Phi(#1 | #2)}

\newcommand{\free}{F}
\newcommand{\inter}{I_{\Z}}

\newcommand{\tf}{g}
\newcommand{\hY}{\hat{Y}}
\newcommand{\lws}{\Omega_{(3)}}
\newcommand{\lwt}{\Omega_{(2)}}
\newcommand{\Yu}{Y}

\newcommand{\wx}{x}
\newcommand{\wl}{l}
\newcommand{\wn}{n}
\newcommand{\wm}{m}
\newcommand{\wy}{y}

\newcommand{\1}{{\bf 1}}
\newcommand{\ideal}{J}

\newtheorem{lemma}{Lemma}[section]
\newtheorem{theorem}[lemma]{Theorem}
\newtheorem{proposition}[lemma]{Proposition}
\newtheorem{corollary}[lemma]{Corollary}
\theoremstyle{definition}
\newtheorem{definition}[lemma]{Definition}

\newtheorem{remark}[lemma]{Remark}

\theoremstyle{standard}
\newtheorem*{writing}{}

\title{A generalization of intertwining operators for vertex operator algebras}
\author{Kenichiro Tanabe\footnote{Research was partially supported by the Grant-in-aid
(No. 24540003 and No. 23224001 (S)) for Scientific Research, JSPS.}\\\\
Department of Mathematics\\
Hokkaido University\\
Kita 10, Nishi 8, Kita-Ku, Sapporo, Hokkaido, 060-0810\\
Japan\\\\
ktanabe@math.sci.hokudai.ac.jp}
\date{}
\begin{document}
\maketitle

\abstract{
For a vertex operator algebra $V$,
we generalize the notion of an intertwining operator
among an arbitrary triple of $V$-modules to an arbitrary triple of $\N$-graded weak $V$-modules and study their properties. 
We show a formula for the dimensions of the spaces of these intertwining operators in terms of modules over
the Zhu algebras under some conditions on $\N$-graded weak modules.}



\section{\label{section:introduction}Introduction}

Let $V$ be a vertex operator algebra.
The purpose of this paper is to generalize the notion of an intertwining operator
among an arbitrary triple of $V$-modules to an arbitrary triple of $\N$-graded weak $V$-modules.
In the representation theory of groups or Lie algebras, 
the tensor product of 
two modules is the tensor product vector space whose module structure is defined
by means of a natural coproduct operation, 
and 
an intertwining operator is defined to be a module homomorphism from the tensor product of two modules
to a third module. 
In contrast, in the representation theory of vertex operator algebras,
first the notion of an intertwining operator among an arbitrary triple of modules is defined in \cite[Definition 5.4.1]{FHL}, and 
then the tensor product of two modules is defined by using intertwining operators.
Note that the existence of the tensor product of two modules over a vertex operator algebra is not guaranteed
in general. 
The dimension of the space of all intertwining operators among a triple of modules is called the {\it fusion rule}. It is a natural problem to determine fusion rules for a given vertex operator algebra.
In \cite[Theorem 1.5.3]{FZ}, Frenkel and Zhu give a formula for the fusion rule 
among an arbitrary triple of irreducible modules in terms of modules over the Zhu algebra 
as a generalization of a result in \cite{TK} for WZW models.
Here we have to be careful that \cite[Theorem 1.5.3]{FZ} is correct for rational vertex operator algebras,
however, 
is not correct for non-rational vertex operator algebras in general
as pointed out at the end of \cite[Section 2]{Li}. A modified result is 
given in \cite[Theorem 2.11]{Li}.
For many vertex operator algebras, fusion rules among triples of irreducible modules have been determined
by using \cite[Theorem 1.5.3]{FZ} and \cite[Theorem 2.11]{Li} (see for example
\cite{A1},\cite{A2},\cite{ADL}, \cite{T1}, \cite{TY1}, and \cite{TY2}). 

 The definition of an intertwining operator in \cite[Definition 5.4.1]{FHL} makes sense
even for weak modules, however, is no longer enough.
Actually based on logarithmic conformal filed theories in physics,
in \cite{M1} a generalization of the notion of an intertwining operator, called
a {\it logarithmic intertwining operator},
is given among an arbitrary triple of logarithmic modules by allowing logarithmic terms.
Here a logarithmic module is  an $\N$-graded weak module such that 
each homogeneous space is a generalized $L(0)$-eigenspace.
A generalization of the formula in \cite[Theorem 1.5.3]{FZ} and \cite[Theorem 2.11]{Li} to logarithmic intertwining operators
is given in \cite[Theorem 6.6]{HY}.

The aims of this paper are to introduce a generalization of the notion of an intertwining operator, which I call
a {\it $\Z$-graded intertwining operator} (see Definition \ref{definition:Z-inter}),
among an arbitrary triple of general $\N$-graded weak modules and to study their properties. 
For logarithmic modules, a $\Z$-graded intertwining operator
is essentially the same as a logarithmic intertwining operator
as we will see later in Section \ref{section:log}.
To explain the main idea, we recall a few facts about intertwining operators for (ordinary) modules over
a vertex operator algebra $(V,Y,\1,\omega)$.
For three $V$-modules $W_i=\oplus_{j=0}^{\infty}(W_i)_{\lambda_i+j}$ with 
lowest weight $\lambda_i\in\C$, $i=1,2,3$ and 
an intertwining operator $I(\ , x) :
W_1\otimes_{\C}W_2\rightarrow W_3\{x\}$,
we define an operator $I^{o}(u,x)=\sum_{i\in\C}u^{o}_{i}x^{-i-1}=x^{\lambda_1+\lambda_2-\lambda_3}I(\ ,x)$,
which is already appeared in \cite[(1.5.3)]{FZ},  \cite[Remark 5.4.4]{FHL}, and \cite[(2.12)]{Li}. 
Here $W_3\{x\}=\{\sum_{\alpha\in\C}w_{\alpha}x^{\alpha}\ |\ 
w_{\alpha}\in W_3\ (\alpha\in\C)\}$.
Then, $I^{o}(u,x)$ is a map from $W_1\otimes_{\C}W_2$ to $W_3\db{x}=
\{\sum_{i\in\Z}w_{i}x^{i}\ |\ w_{i}\in W_3\ (i\in\Z)
\mbox{ and } w_{i}=0, i\ll 0\}$
and $I(\ ,x)$ can be written
as
\begin{align}
\label{eq:conj-L(0)}
I(u,x)v&=\sum_{i\in\Z}x^{L(0)}(x^{-L(0)}u)_i^{o}x^{-L(0)}v
\end{align}
for $u\in W_1$ and $v\in W_2$.
Here for the coefficient $L(0)$ of $x^{-2}$ in each $Y_{W_i}(\omega,x)$, $i=1,2,3$,
we define
\begin{align}
\label{eq:xpmL(0)w}
x^{\pm L(0)}w&=x^{\pm \lambda}w
\end{align}
for $w\in W_i$ with $L(0)w=\lambda w, \lambda\in\C$
and extend $x^{\pm L(0)}w$ for an arbitrary $w\in W_i,i=1,2,3$ by linearity. 
We note that $I^{o}(\ ,x)$ satisfies all the conditions in the definition of 
intertwining operator \cite[Definition 5.4.1]{FHL} except the so called $L(-1)$-derivative property.
For (ordinary) $V$-modules, 
$I^{o}(\ ,x)$
is nothing but a $\Z$-graded intertwining operator
as we will see later in Proposition \ref{proposition:ord-Z-graded}.
The main idea is that we 
redefine $x^{\pm L(0)}$ to be formal variables such that
\begin{align}
x\dfrac{d}{dx}x^{\pm L(0)}=x^{\pm L(0)}(\pm L(0)).
\end{align}
Since this definition of 
$x^{\pm L(0)}$ makes sense for $\N$-graded weak modules,
we can define \lq\lq intertwining operators\rq\rq for $\N$-graded weak $V$-modules by 
using $I^{o}(\ ,x)$ and \eqref{eq:conj-L(0)}.
Moreover, applying $xd/dx$ to both sides of \eqref{eq:conj-L(0)} and 
using Borcherds identity, we automatically get the $L(-1)$-derivative property.

We expect that various results for intertwining operators
can be generalized to $\Z$-graded intertwining operators.
As the main result of this paper
I will show 
a formula for the fusion rules  as a 
generalization of  \cite[Theorem 1.5.3]{FZ} and \cite[Theorem 2.11]{Li} in Theorem \ref{theorem:correspondence}. 
To state the result precisely, we prepare following symbols.
For a vertex operator algebra $V$ and a weak $V$-module $W$,
$A(V)$ is the Zhu algebra defined in \cite[Section 2.1]{Z} 
and $A(W)$ is the $A(V)$-bimodule defined in \cite[Theorem 1.5.1]{FZ}.
For a left $A(V)$-module $U$, 
$S(U)=\oplus_{j=0}^{\infty}S(U)(j)$ is the generalized Verma module with 
$S(U)(0)=U$ given in \cite{DLM1}, $U^{*}=\Hom_{\C}(U,\C)$, 
$S(U)^{\prime}=\oplus_{j=0}^{\infty}\Hom_{\C}(S(U)(j),\C)$，and 
$\inter\binom{S(\lws^{*})^{\prime}}{W_1\ S(\lwt)}$
is the space of all $\Z$-graded intertwining operators of type 
$\binom{S(\lws^{*})^{\prime}}{W_1\ S(\lwt)}$.
Now we state the main result:
\begin{writing}[Theorem \ref{theorem:correspondence}]
For  an $\N$-graded weak $V$-module $W_1$ and
two left $A(V)$-modules $\Omega_{(2)}$ and $\Omega_{(3)}$,
the  map
\begin{align*}
\inter\binom{S(\lws^{*})^{\prime}}{W_1\ S(\lwt)}
&\rightarrow \Hom_{A(V)}(A(W_1)\otimes_{A(V)}\lwt,\lws)\nonumber\\
\Phi(\ ,x)&\mapsto o^{\Phi}& \eqref{eq:one-one-iinter-N}
\end{align*}
is a linear isomorphism.
\end{writing}

As a direct consequence of Theorem \ref{theorem:correspondence}, we
obtain the following result.
\begin{writing}[Corollary \ref{corollary:correspondence}]
	Let $W_i=\oplus_{j=0}^{\infty}W_i(j), i=1,2,3$ be three $\N$-graded weak $V$-modules such that
	$W_2$ and $W_3^{\prime}$ are generalized Verma $V$-modules 
	and $\dim_{\C}W_3(j)<\infty$ for all $j\in\N$.
	Then, the  map
	\begin{align*}
	\inter\binom{W_3}{W_1\ W_2}
	&\rightarrow \Hom_{A(V)}(A(W_1)\otimes_{A(V)}W_2(0),W_3(0))\nonumber\\
	\Phi(\ ,x)&\mapsto o^{\Phi}&\eqref{eq:one-one-iinter-N-c}
	\end{align*}
	is a linear isomorphism.
\end{writing}

\noindent{}Here we define 
$o^{\Phi}(u)=\Phi(u;\deg u-1)$ for homogeneous $u\in W_1$ and
extend $o^{\Phi}(u)$ for an arbitrary $u\in W_1$ by linearity(see \eqref{eq:inter-zhu}). 
To show the main result, we will modify the proofs of \cite[Theorem 6.6]{HY}, \cite[Theorem 2.11]{Li}, 
and \cite[Theorem 2.2.1]{Z} so as not to  use the $L(-1)$- derivative property.
For some vertex operator algebras and their modules,
we can compute the right-hand side of \eqref{eq:one-one-iinter-N}.
For instance, let us consider Verma modules $M_{c,h}$, $c,h\in\C$ over the Virasoro vertex operator algebra
$M_{c}$ where we use the notation in \cite[Section 4]{FZ}.
In this case, the same computation as in \cite[Section 2]{Li} shows
that the right-hand side of \eqref{eq:one-one-iinter-N} reduces to the following simple form: 
\begin{align*}
\Hom_{A(M_{c})}(A(M_{c,h})\otimes_{A(M_{c})}\lwt,\lws)&\cong\Hom_{\C}(\lwt,\lws).
\end{align*}

The organization of the paper is as follows. 
In Section \ref{section:preliminary} we recall some basic properties of 
the Zhu algebras, bimodules over the Zhu algebras, and $\N$-graded weak modules.
In Section \ref{section:Z-inter} we recall some  basic facts about intertwining operators and introduce the notion of 
a $\Z$-graded intertwining operator.
In Section \ref{section:log} for an arbitrary  triple of logarithmic modules $W_i, i=1,2,3$,
we construct a linear 
isomorphism from the space of all logarithmic intertwining operators of type $\binom{W_3}{W_1\ W_2}$
to the space of all $\Z$-graded intertwining operators of the same type.
In Section \ref{section:correspondence} we will show the main result.
In Section \ref{section:notation} we list some notations.

\section{\label{section:preliminary}Preliminary}
We assume that the reader is familiar with the basic knowledge on
vertex algebras as presented in \cite{B}, \cite{FLM}, and \cite{LL}. 
Throughout this paper, $\N$ denotes the set of all non-negative integers,
$x,y,x_0,x_1,x_2,\ldots$ are commutative formal variables, and
$(V,Y,{\mathbf 1},\omega)$ is a vertex operator algebra.
For the Virasoro element $\omega$ of $V$ and a weak $V$-module $(M,Y_{M})$, we write
\begin{align}
Y_{M}(\omega,x)=\sum_{i\in\Z}L(i)x^{-i-2}.
\end{align}

We recall some properties of the Zhu algebra $A(V)$ of $V$ and  
the $A(V)$-bimodules associated with weak modules from \cite[Section 2]{Z} and \cite[Section 1]{FZ}, and \cite[Section 2]{Li}.
Let $M$ be a weak $V$-module. 
For homogeneous $a\in V$ and $u\in M$, we define
\begin{align}
\label{eq:zhu-ideal-multi}
a\circ u&=\Res_{x}(1+x)^{\wt a}x^{-2}Y_{M}(a,x)u\in M
\end{align}
and 
\begin{align}
\label{eq:zhu-bimodule-left}
a*u&=\Res_{x}(1+x)^{\wt a}x^{-1}Y_{M}(a,x)u\in M,\\
\label{eq:zhu-bimodule-right}
u*a&=\Res_{x}(1+x)^{\wt a-1}x^{-1}Y_{M}(a,x)u\in M.
\end{align}
Here, $\Res_{x}$ is defined by
\begin{align}
\Res_{x}f(x)=f_{-1}
\end{align}
for $f(x)=\sum_{i\in\Z}f_ix^i\in M[[x,x^{-1}]]$.
We extend \eqref{eq:zhu-ideal-multi}--\eqref{eq:zhu-bimodule-right} for an arbitrary $a\in V$ by linearity.
We also define
\begin{align}
\label{eq:zhu-bimodule-ideal}
O(M)&=\Span_{\C}\{a\circ u\ |\ a\in V, u\in M\}
\end{align}
and take the following quotient space:
\begin{align}
\label{eq:zhu-bimodule}
A(M)&=M/O(M).
\end{align}
If one takes $M=V$, then $A(V)$ is an associative $\C$-algebra,
 called the {\em Zhu algebra} of $V$, with multiplication  
\eqref{eq:zhu-bimodule-left} by \cite[Theorem 2.1.1]{Z}.
It follows from \cite[Theorem 1.5.1]{FZ} that
$A(M)$ is an $A(V)$-bimodule under the actions \eqref{eq:zhu-bimodule-left} and \eqref{eq:zhu-bimodule-right}.
It follows from the proof of \cite[Lemma 2.1.2]{Z} that for homogeneous $a\in V$ and $u\in M$,
\begin{align}
\label{eq:leq-2}
\Res_{x}(1+x)^{\wt a}x^{k}Y_{M}(a,x)u\in O(M)\mbox{ for }k\leq -2.
\end{align}
We shall use elements of $M$ to represent elements of $A(M)$.
For a left $A(V)$-module $U$ and $a\in A(V)$, we denote the action of $a$ on $U$ by $o(a)$:
\begin{align}
A(V) &\rightarrow \End_{\C}(U)\nonumber\\
a&\mapsto o(a).
\end{align}

A weak $V$-module $M$ is called {\it $\N$-graded}
if $M$ admits a decomposition $M=\oplus_{j=0}^{\infty}M(j)$ such that
\begin{align}
a_{k}M(j)\subset M(\wt a+j-k-1)
\end{align}
for homogeneous $a\in V$, $j\in \N$, and $k\in\Z$.
For an $\N$-graded weak $V$-module $M=\bigoplus_{j=0}^{\infty}M(j)$ and 
$u\in M(j), j\in\N$, we define the {\it degree} of $u$ by
\begin{align}
\deg u&=j.
\end{align}
Following \cite{DLM1},
an $\N$-graded weak $V$-module $M=\oplus_{i=0}^{\infty}M(i)$ is called 
a {\em generalized Verma $V$-module} if 
$M$ is generated by $M(0)$ and 
for every $\N$-graded weak $V$-module $W$ and
every $A(V)$-module homomorphism 
\begin{align}
f : M(0)\rightarrow \{w\in W\ |\ 
a_{i}w=0\mbox{ for homogeneous $a\in V$ and $i\geq \wt a$}\},
\end{align}
there exists a unique $V$-module homomorphism
$F : M\rightarrow W$ such that $F|_{M(0)}=f$. For an arbitrary $A(V)$-module $U$, \cite[Theorem 6.2]{DLM1} shows
 there exists a unique generalized Verma $V$-module $S(U)$ with $S(U)(0)=U$ up to isomorphism,
where $S(U)$ is denoted by $\bar{M}(U)$ in \cite{DLM1}.

\section{$\Z$-graded intertwining operators\label{section:Z-inter}
}
In this section we first recall the definition of an intertwining operator from \cite[Definition 5.4.1]{FHL}
and then introduce the notion of a $\Z$-graded intertwining operator
as a generalization of an intertwining operator.
For a vector space $M$ over $\C$ and $p,q\in\Z$, we define
\begin{align}
M[\wx,\wx^{-1}]_{[p,q]}&=\{\sum_{i=p}^{q}u_{i}x^{i}\ |\ u_p,u_{p+1},\ldots,u_{q}\in M\},\nonumber\\
M\{x\}&=\{\sum_{\alpha\in\C}u_{\alpha}x^{\alpha}\ |\ u_{\alpha}\in M\ 
(\alpha\in\C)\},\nonumber\\
M\db{x}&=\{\sum_{i\in\Z}u_{i}x^{i}\ |\ u_{i}\in M\ (i\in\Z)\mbox{ and }u_{i}=0, i\ll 0\},\mbox{ and}\nonumber\\
M[[x,y]]&=\{\sum_{i,j=0}^{\infty}u_{ij}x^{i}y^{j}\ |\ u_{ij}\in M\ (i,j\in\N)\}.
\end{align}
Define three linear injective maps
\begin{align}
\iota_{\wx,y} : & M[[\wx,y]][\wx^{-1},y^{-1},(\wx-y)^{-1}]\rightarrow 
M\db{x}\db{y},\nonumber\\
\iota_{y,\wx} : &M[[\wx,y]][\wx^{-1},y^{-1},(\wx-y)^{-1}]
\rightarrow M\db{y}\db{x},\mbox{ and}\nonumber\\
\iota_{\wx,\wy-\wx} : & M[[\wx,y]][\wx^{-1},y^{-1},(\wx-y)^{-1}]
\rightarrow M\db{\wx}\db{\wy-\wx}
\end{align}
determined by 
\begin{align}
\label{eqn:expand}
\iota_{\wx,y}\big(\wx^jy^k(\wx-y)^l\big)
&=\sum_{i=0}^{\infty}\binom{l}{i}(-1)^{i}\wx^{j+l-i}y^{k+i},\nonumber\\
\iota_{y,\wx}\big(\wx^jy^k(\wx-y)^l\big)
&=\sum_{i=0}^{\infty}\binom{l}{i}(-1)^{l-i}y^{k+l-i}\wx^{j+i},\mbox{ and}\nonumber\\
\iota_{\wx,\wy-\wx}
\big(\wx^jy^k(\wx-y)^l\big)&=\sum_{i=0}^{\infty}\binom{k}{i}{\wx}^{j+k-i}(-1)^{l}(\wy-\wx)^{l+i}
\end{align}
for $j,k,l\in\Z$ and $\iota_{x,y}(u)=\iota_{y,x}(u)=\iota_{\wx,\wy-\wx}(u)=u$ for $u\in M$.

We recall the definition of an intertwining operator
from \cite[Definition 5.4.1]{FHL}.
\begin{definition}
\label{definition:intertwining-operator}
Let $W_1, W_2$, and $W_3$ be three $V$-modules. 
An {\em intertwining operator} of type $\binom{W_3}{W_1\ W_2}$ is a linear map 
\begin{align}
\label{eq:inter-form}
I(\ , x) : W_1\otimes_{\C}W_2&\rightarrow W_3\{x\}\nonumber\\
I(\ewo, x)\ewt&=\sum_{\alpha\in\C}\ewo_{\alpha}\ewt x^{-\alpha-1},\nonumber\\
&\ewo\in W_1,\ewt\in W_2, \mbox{ and }\ewo_{\alpha}\in \Hom_{\C}(W_2,W_3),
\end{align}
such that the following conditions are satisfied:
\begin{enumerate}
\item
For $\ewo\in W_1,\ewt\in W_2$, and $\alpha\in\C$,
\begin{align}
\label{eq:inter-truncation}
\ewo_{\alpha+m}\ewt&=0\mbox{ for sufficiently large $m\in\N$.}
\end{align}

\item
For $\ewo\in W_1$ and $a\in V$,
\begin{align}
\label{eq:inter-borcherds}
&x_0^{-1}\delta(\dfrac{x_1-x_2}{x_0})Y(a,x_1)I(\ewo,x_2)-
x_0^{-1}\delta(\dfrac{x_2-x_1}{-x_0})I(\ewo,x_2)Y(a,x_1)\nonumber\\
&=x_1^{-1}\delta(\dfrac{x_2+x_0}{x_1})I(Y(a,x_0)\ewo,x_2).
\end{align}
\item ($L(-1)$-derivative property)
For $\ewo\in W_1$,
\begin{align}
\label{eq:inter-derivative}
I(L(-1)\ewo,x)&=\dfrac{d}{dx}I(\ewo,x).
\end{align}\end{enumerate}
\end{definition}
We denote  by $I\binom{W_3}{W_1\ W_2}$
the space of all intertwining operators of type $\binom{W_3}{W_1\ W_2}$
and call its dimension the {\it fusion rule} of type $\binom{W_3}{W_1\ W_2}$. 
A formula for the fusion rule among an arbitrary triple of irreducible modules is 
given in \cite[Theorem 1.5.3]{FZ} and \cite[Theorem 2.11]{Li}.

For a vector space $U$, we define a subspace $U\lfloor \{x\}\rfloor$ of $U\{x\}$ by
\begin{align}
U\lfloor \{x\}\rfloor
&=\Big\{\sum_{\alpha\in\C}u_{\alpha}x^{\alpha}
\ \Big|\ 
\mbox{\begin{tabular}{l}
$u_{\alpha}\in U\ (\alpha\in\C)$ and for any $\alpha\in\C$,\\
$w_{\alpha+i}=0$ for sufficiently small $i\in\Z$
\end{tabular}}
\Big\}.
\end{align}
Standard arguments (cf. \cite[Sections 3.2--3.4]{LL} and \cite[Lemma 2.4]{T2}) show that the condition (2) in Definition \ref{definition:intertwining-operator} is
equivalent to the following condition:
For $u\in W_1, v\in W_2$ and $a\in V$,
there exists
\begin{align}
\label{eq:inter-global}
I(a,u,v|x,y)\in 
W_3[[y]]\lfloor \{x\}\rfloor[\wy^{-1},(\wx-y)^{-1}]
\end{align}
such that 
\begin{align}
\iota_{\wx,y}I(a,u,v|x,y)&=I(u,x)Y_{W_2}(a,y)v,\nonumber\\
\iota_{\wy,\wx}I(a,u,v|x,y)&=Y_{W_3}(a,y)I(u,x)v,\quad\mbox{and }\nonumber\\
\iota_{\wx,\wy-\wx}I(a,u,v|x,y)&=I(Y_{W_1}(a,y-x)u,x)v.
\end{align}
Let $
I(\ , x) : W_1\otimes_{\C}W_2\rightarrow W_3\{x\}$ be an intertwining operator.
For $\alpha\in\C$, by taking $a$ to be the Virasoro element $\omega\in V$ in \eqref{eq:inter-borcherds}
and comparing the 
coefficients of $x_0^{-1}x_1^{-2}x_2^{-\alpha-1}$ in both sides, we have
\begin{align}
\label{eq:relation-borcherds-L(0)}
(L(-1)\ewo)_{\alpha+1}
&=L(0)\ewo_{\alpha}-(L(0)\ewo)_{\alpha}-\ewo_{\alpha}L(0)
\end{align}
and therefore the $L(-1)$-derivative property \eqref{eq:inter-derivative} can be replaced by the following condition:
\begin{align}
\label{eq:relation-L(0)}
x\dfrac{d}{dx}I(\ewo,x)&=
L(0)I(\ewo,x)-I(L(0)\ewo,x)-I(\ewo,x)L(0).
\end{align}
Suppose $W_i, i=1,2,3$ admit decompositions
\begin{align}
\label{eq:v-module-decomposition}
W_i&=\bigoplus_{j=0}^{\infty}(W_i)_{\lambda_i+j}
\end{align}
where $(W_i)_{\lambda_i+j}$ is the eigenspace for $L(0)$ with eigenvalue $\lambda_i+j, j\in\N$.
It follows from \eqref{eq:relation-L(0)} that
\begin{align}
\label{eq:restriction-index}
u_{\alpha}v&=0 \mbox{ for }\alpha\not\in \lambda_1+\lambda_2-\lambda_3+\Z
\end{align}
and
\begin{align}
\label{eq:shift-grading}
u_{\lambda_1+\lambda_2-\lambda_3+k}
(W_2)_{\lambda_2+j} & \subset (W_3)_{\lambda_3+i+j-k-1}
\end{align} 
for $u\in (W_1)_{\lambda_1+i}$, $j\in\N$, and $k\in\Z$.
The properties \eqref{eq:restriction-index} and 
\eqref{eq:shift-grading} are essentially used in the proof of the formula  for the 
fusion rules 
given in \cite[Theorem 1.5.3]{FZ} and \cite[Theorem 2.11]{Li}.

The definition of an intertwining operator in \cite[Definition 5.4.1]{FHL} makes sense
even for weak $V$-modules, however, the condition \eqref{eq:relation-L(0)} seems to be too strong for weak $V$-modules
as explained below.
Let $W_i,i=1,2,3$ be three weak $V$-modules
and $I(\ ,x) : W_1\otimes W_2\rightarrow W_3\{x\}$ a linear map which satisfies all the conditions in
Definition \ref{definition:intertwining-operator}.
Despite the action of $L(0)$ on a weak $V$-module is not necessarily semisimple, 
for $u\in W_1$ and $v\in W_2$, 
\eqref{eq:relation-L(0)} forces that each coefficient of
\begin{align}
L(0)I(u,x)v-I(L(0)\ewo,x)v-I(\ewo,x)L(0)v
\end{align} 
is a scalar multiple of the corresponding coefficient of $I(u,x)v$. 
Moreover, we can not expect a generalization of the formula for the
fusion rules given in \cite[Theorem 1.5.3]{FZ} and \cite[Theorem 2.11]{Li}
because similar conditions like \eqref{eq:restriction-index} and \eqref{eq:shift-grading},
which are essential for these results, do not follow from \eqref{eq:relation-L(0)}. 
Thus, we need to modify the condition \eqref{eq:relation-L(0)}.

To do that, we return to the case of intertwining operators $I(\ ,x)$ among a triple of 
$V$-modules $W_1,W_2$ and $W_3$ as in \eqref{eq:v-module-decomposition}.
We define a map
\begin{align}
I^{o}(\ ,x)=x^{\lambda_1+\lambda_2-\lambda_3}I(\ ,x)
: W_1\otimes_{\C}W_2 & \rightarrow W_3\db{x}\nonumber\\
\ewo\otimes \ewt & \mapsto \sum_{i\in\Z}\ewo^{o}_{i}x^{-i-1},
\end{align}
which is already appeared in \cite[(1.5.3)]{FZ},  \cite[Remark 5.4.4]{FHL}, and \cite[(2.12)]{Li}, 
and we denote $(W_i)_{\lambda_i+j}$ by $W_i(j)$ for $i=1,2,3$ and $j\in\N$.
Then, we  have
\begin{align}
\label{eq:conj-L(0)-re}
I(u,x)v&=x^{-\lambda_1-\lambda_2+\lambda_3}I^{o}(u,x)v\nonumber\\
&=\sum_{i\in\Z}x^{L(0)}(x^{-L(0)}u)_i^{o}x^{-L(0)}v
\end{align}
for $u\in W_1$ and $v\in W_2$.
Here we define 
\begin{align}
\label{eq:xpmL(0)w-re}
x^{\pm L(0)}w&=x^{\pm \lambda}w
\end{align}
for $w\in W_i,i=1,2,3$ with $L(0)w=\lambda w, \lambda\in\C$
and extend $x^{\pm L(0)}w$ for an arbitrary $w\in W_i,i=1,2,3$ by linearity. 
The map $I^{o}(\ ,x)$ 
satisfies \eqref{eq:inter-truncation},
\eqref{eq:inter-borcherds}, and 
\begin{align}
\label{eq:shift-N-grading}
u^{o}_{k}W_2(j) & \subset W_3(i+j-k-1)
\end{align}
for $u\in W_1(i)$, $j\in\N$, and $k\in\Z$ by \eqref{eq:shift-grading}.
Based on the properties \eqref{eq:inter-truncation},
\eqref{eq:inter-borcherds}, and \eqref{eq:shift-N-grading} of
$I^{o}(\ ,x)$, we introduce the following notion:
\begin{definition}\label{definition:Z-inter}
Let $W_1,W_2$ and $W_3$ be three $\N$-graded weak $V$-modules
and 
$\Phi(\ ,x)=\sum_{n\in\Z}\Phi(\ ;n)x^{-n-1}$ a linear map from $W_1\otimes_{\C} W_2$ to $W_3\db{\wx}$.
We call $\Phi$ a {\it $\Z$-graded intertwining operator} of type $\binom{W_3}{W_1\ W_2}$
if 
\begin{enumerate}
\item
For $i,j\in\N$, $k\in\Z$,  and $u\in W_1(i)$, 
\begin{align}
\label{eq:N-shift}
\Phi(u;k)W_2(j)\subset W_3(i+j-k-1).
\end{align}
\item For $u\in W_1, v\in W_2$, and $a\in V$,
there exists
\begin{align}
\label{eq:N-global}
\glb{a,u,v}{\wx,\wy}\in 
W_{3}[[\wx,y]][\wx^{-1},y^{-1},(\wx-y)^{-1}]
\end{align}
such that 
\begin{align}
\iota_{\wx,y}\glb{a,u,v}{\wx,\wy}&=\Phi(u,x)Y_{W_2}(a,y)v,\nonumber\\
\iota_{\wy,\wx}\glb{a,u,v}{\wx,\wy}&=Y_{W_3}(a,y)\Phi(u,x)v,\quad\mbox{and }\nonumber\\
\iota_{\wx,\wy-\wx}\glb{a,u,v}{\wx,\wy}&=\Phi(Y_{W_1}(a,y-x)u,x)v.
\end{align}
\end{enumerate}
\end{definition}
Standard arguments (cf. \cite[Sections 3.2--3.4]{LL} and \cite[Lemma 2.4]{T2}) show that the condition (2) in Definition \ref{definition:Z-inter} is
equivalent to the following {\em Borcherds identity}: for $a\in V$, $\ewo\in W_1$, $\ewt \in W_2$,
and $l,m,n\in\Z$, we have
\begin{align}
\label{eq:borcherds-id-N}
&\sum\limits_{i=0}^{\infty}
\binom{m}{i}\Phi(a_{l+i}\ewo;m+n-i)\ewt\nonumber\\
&=\sum\limits_{i=0}^{\infty}
(-1)^i\binom{l}{i}\big(
a_{m+l-i}\Phi(\ewo;n+i)\ewt+(-1)^{l+1}\Phi(\ewo;m+i)a_{n+l-i}\ewt\big).
\end{align}
We denote by $I_{\Z}\binom{W_3}{W_1\ W_2}$
the space of all $\Z$-graded intertwining operators of type $\binom{W_3}{W_1\ W_2}$.
The following result shows that for $V$-modules
a $\Z$-graded intertwining operator is essentially the same as
an intertwining operator.
\begin{proposition}
\label{proposition:ord-Z-graded}
For three $V$-modules $W_1$, $W_2$, and $W_3$, the map 
\begin{align}
\label{eq:one-one-inter}
I\binom{W_3}{W_1\ W_2}&\rightarrow I_{\Z}\binom{W_3}{W_1\ W_2}\nonumber\\
I(\ ,x)&\mapsto I^{o}(\ , x)
\end{align}
is a linear isomorphism.
\end{proposition}
\begin{proof}
We may assume that
$W_1$, $W_2$, and $W_3$ admit decompositions as in \eqref{eq:v-module-decomposition}.
We have already shown before Definition \ref{definition:Z-inter} that
$I^{o}(\ ,x)$ is a $\Z$-graded intertwining operator for an intertwining operator $I(\ ,x)$.
For a $\Z$-graded intertwining operator
 $\Phi(\ , x) : 
W_1\otimes_{\C}W_2\rightarrow W_3\db{x}$,
$x^{-\lambda_1-\lambda_2+\lambda_3}\Phi(\ ,x)$ satisfies  
\eqref{eq:inter-truncation} and \eqref{eq:inter-borcherds}. The same argument as in
\eqref{eq:relation-borcherds-L(0)} shows 
\begin{align}
\label{eq:relation-N-borcherds-L(0)}
\Phi(L(-1)\ewo;i+1)
&=L(0)\Phi(\ewo;i)-\Phi(L(0)\ewo;i)-\Phi(\ewo;i)L(0)
\end{align}
for $\ewo\in W_1$ and $i\in\Z$, which implies
\begin{align}
\dfrac{d}{dx}(x^{-\lambda_1-\lambda_2+\lambda_3}\Phi(u ,x)) =x^{-\lambda_1-\lambda_2+\lambda_3}\Phi(L(-1)u,x) 
\end{align}
by \eqref{eq:N-shift}. Therefore
$x^{-\lambda_1-\lambda_2+\lambda_3}\Phi(u ,x)$
is an intertwining operator and this completes the proof.
\end{proof}
As we will see later in Proposition \ref{relation-log-N}, the isomorphism \eqref{eq:one-one-inter}
 is generalized to 
the case of logarithmic intertwining operators introduced in \cite{M1}.

We note that the $L(-1)$-derivative property \eqref{eq:inter-derivative}, or equivalently \eqref{eq:relation-L(0)},
is not required for $\Z$-graded intertwining operators. 
However, the following modifications of $\Z$-graded intertwining operators
satisfy \eqref{eq:relation-L(0)}.
We redefine $x^{L(0)}$ and $x^{-L(0)}$ to be two formal variables
and let $\C x^{L(0)}$ (resp. $\C x^{-L(0)}$) be a vector space with a basis $x^{L(0)}$
(resp. $x^{-L(0)}$). 
For a $L(0)$-module $W$,
we define vector spaces
\begin{align}
x^{L(0)}W&=\C x^{L(0)}\otimes_{\C}W,\nonumber\\
x^{-L(0)}W&=\C x^{-L(0)}\otimes_{\C}W\end{align}
and a linear map
\begin{align}
\label{eq;def-xd/dx}
x\dfrac{d}{dx} :\ & x^{\pm L(0)}W\rightarrow x^{\pm L(0)}W\nonumber\\
&x^{\pm L(0)}\otimes u\mapsto x^{\pm L(0)}\otimes(\pm L(0)u),\quad u\in W.
\end{align}
If $W$ is a weak $V$-module, then so are $x^{\pm L(0)}W$ by defining
\begin{align}
Y_{x^{\pm L(0)}W}(a,y)(x^{\pm L(0)}\otimes u)&=x^{\pm L(0)}\otimes Y_{W}(a,y)u
\end{align}
for $a\in V$ and $u\in W$.
Clearly $x^{\pm L(0)}W$ are isomorphic to $W$.
For three $L(0)$-modules $W_i,i=1,2,3$ and a linear map 
$f: x^{-L(0)}W_1\otimes_{\C}x^{-L(0)}W_2\rightarrow x^{L(0)}W_3$, we define
a map
\begin{align}
x\dfrac{d}{dx} f : x^{-L(0)}W_1\otimes_{\C}x^{-L(0)}W_2\rightarrow x^{L(0)}W_3
\end{align}
by 
\begin{align}
\label{eq;def-hom-xd/dx}
&(x\dfrac{d}{dx}f)(p\otimes q)\nonumber\\
&=x\frac{d}{dx}(f(p\otimes q))+
f((x\frac{d}{dx}p)\otimes q)+f(p\otimes x\frac{d}{dx}q)
\end{align}
for $p\in x^{-L(0)}W_1$ and $q\in x^{-L(0)}W_2$.
For three $\N$-graded weak $V$-modules $W_1, W_2$, and $W_3$,
a $\Z$-graded intertwining operator $\Phi : W_1\otimes_{\C}W_2\rightarrow W_3\db{x}$, and $i\in\Z$, we define
a map 
\begin{align}
\label{eq:modify-inter}
\hat{\Phi}_{i} : x^{-L(0)}W_1\otimes_{\C}x^{-L(0)}W_2&\rightarrow x^{L(0)}W_3\nonumber\\
(x^{-L(0)}\otimes u)\otimes (x^{-L(0)}\otimes v)&\mapsto x^{L(0)}\otimes \Phi(u;i)v.
\end{align}
Then the sequence $(\hat{\Phi}_i)_{i\in\Z}$ satisfies \eqref{eq:borcherds-id-N} and
\begin{align}
\hat{\Phi}_k((x^{-L(0)}\otimes u) \otimes (x^{-L(0)}\otimes v))\in x^{L(0)}W_3(i+j-k)
\end{align}
for $k\in\Z$, $u\in W_1(i)$, and $v\in W_2(j)$, which is an analogue of 
\eqref{eq:N-shift}.
By \eqref{eq;def-hom-xd/dx}, we automatically have  
the following analogue of the $L(-1)$-derivative property 
\eqref{eq:inter-derivative} (or \eqref{eq:relation-L(0)}).
\begin{lemma}
For $i\in\Z$, $\ewo\in W_1$, and $\ewt\in W_2$, we have
\begin{align}
(x\dfrac{d}{dx}\hat{\Phi}_{i})(\ewo\otimes \ewt)&
=
L(0)\hat{\Phi}_{i}(\ewo\otimes \ewt)-
\hat{\Phi}_{i}(L(0)\ewo\otimes \ewt)-
\hat{\Phi}_{i}(\ewo\otimes L(0)\ewt).
\end{align}
\end{lemma}

\section{\label{section:log}A relation between logarithmic intertwining operators and 
$\Z$-graded intertwining operators}
In this section we will show 
that for logarithmic modules, a $\Z$-graded intertwining operator
is essentially the same as a logarithmic intertwining operator introduced in \cite{M1}.
Throughout this section we assume all weak $V$-modules $M$ satisfy 
the following condition: there exists $\lambda\in\C$ such that
$M$ admits a decomposition
\begin{align}
\label{eq:finite-module-1}
&M=\bigoplus_{i=0}^{\infty}M_{\lambda+i},\nonumber\\
&\quad
M_{h}=\{u\in M\ |\ (L(0)-h)^{n}u=0\mbox{ for some $n\in\Z_{>0}$}\}\nonumber\\
&\quad\mbox{ with }
\dim_{\C}M_{h}<\infty\mbox{ for $h\in \lambda+\Z$.}  
\end{align}
Finite direct sums of weak $V$-modules satisfying the condition above are called 
{\em logarithmic} $V$-modules in \cite{M1}. 
We recall the definition of logarithmic intertwining operators from 
\cite[Definition 1.3]{M1} and \cite[Definition 3.7]{HLZ}.
\begin{definition}
\label{definition:log}
Let $W_i=\oplus_{j=0}^{\infty}(W_i)_{\lambda_i+j},i=1,2,3$ be three weak $V$-modules which satisfy
\eqref{eq:finite-module-1}.
A {\em logarithmic intertwining operator} is a linear map 
\begin{align}
\label{eq:log-form}
I(\ , x) : W_1\otimes_{\C}W_2&\rightarrow W_3[\log x]\{x\}\nonumber\\
I(\ewo, x)\ewt&=\sum_{\alpha\in\C}\sum_{n=0}^{\infty}\ewo_{\alpha,n}\ewt x^{-\alpha-1}(\log x)^{n},\nonumber\\
&\ewo\in W_1,\ewt\in W_2, \mbox{ and }\ewo_{\alpha,n}\in \Hom_{\C}(W_2,W_3)
\end{align}
such that the following conditions are satisfied:
\begin{enumerate}
\item
For $\ewo\in W_1,\ewt\in W_2$, and $\alpha\in\C$,
\begin{align}
\label{eq:truncation}
\ewo_{\alpha+m,k}\ewt&=0\mbox{ for $m\in\N$ sufficiently large, independently of $k$.}
\end{align}

\item
For $\ewo\in W_1$ and $a\in V$,
\begin{align}
\label{eq:log-borchrerds}
&x_0^{-1}\delta(\dfrac{x_1-x_2}{x_0})Y(a,x_1)I(\ewo,x_2)-
x_0^{-1}\delta(\dfrac{x_2-x_1}{-x_0})I(\ewo,x_2)Y(a,x_1)\nonumber\\
&=x_1^{-1}\delta(\dfrac{x_2+x_0}{x_1})I(Y(a,x_0)\ewo,x_2).
\end{align}
\item
For $\ewo\in W_1$,
\begin{align}
\label{eq:derivative-property}
I(L(-1)\ewo,x)&=\dfrac{d}{dx}I(\ewo,x).
\end{align}\end{enumerate}
\end{definition}
We denote by $I_{\log}\binom{W_3}{W_1\ W_2}$
the space of all logarithmic intertwining operators of type $\binom{W_3}{W_1\ W_2}$.
We recall some basic properties about logarithmic intertwining operators from \cite{HLZ},\cite{M1}, and \cite{M2}.
The same argument as in \eqref{eq:restriction-index}
shows
\begin{align}
\ewo_{\alpha,n}\ewt=0 \mbox{ for }\alpha\not\in \lambda_1+\lambda_2-\lambda_3+\Z.
\end{align}
For a logarithmic intertwining operator $I(\ ,x) : W_1\otimes_{\C}W_2\rightarrow W_3[\log x]\{x\}$,
we write 
\begin{align}
I(\ ,x)&=\sum_{i=0}^{\infty}I^{(i)}(\ ,x)(\log x)^i,&I^{(i)}(\ ,x) : W_1\otimes_{\C}W_2\rightarrow W_3\{x\}
\end{align}
and define
\begin{align}
I^{o}(\ ,x)&=\sum_{i\in\Z}I^{o}(\ ;i) x^{-i-1},\ 
I^{o}(\ ;i)\in \End_{\C}(W_1\otimes_{\C}W_2,W_3)\nonumber\\
&=x^{\lambda_1+\lambda_2-\lambda_3}I^{(0)}(\ ,x)
\in (\End_{\C}(W_2,W_3))[[x,x^{-1}]].
\end{align}
Note that
\begin{align}
I^{o}(\ewo;i)&=\ewo_{i+\lambda_1+\lambda_2-\lambda_3,0}
\end{align}
for $\ewo\in W_1$ and $i\in\Z$.
It follows from \eqref{eq:derivative-property} that
\begin{align}
\label{eq:inductive-inter}
xI^{(i)}(L(-1)\ewo ,x)&=x\dfrac{d}{dx}I^{(i)}(\ewo ,x)+
(i+1)I^{(i+1)}(\ewo ,x)
\end{align}
for all $i\in\N$ and therefore $I(\ ,x)$ is uniquely determined by $I^{(0)}(\ ,x)$,
or $I^{o}(\ ,x)$.

For a weak $V$-module $M=\bigoplus_{i=0}^{\infty}M_{\lambda+i}$ as in \eqref{eq:finite-module-1}, 
$M$ is an $\N$-graded weak $V$-module with $M(i)=M_{\lambda+i}$ for $i\in\N$ and one can take 
the Jordan decomposition 
\begin{align}
\label{eq:jordan}
L(0)&=S+N
\end{align}
of $L(0)$ on $M$ where $S$ is the semisimple part of $L(0)$
and $N$ is the nilpotent part of $L(0)$.
For $u\in M$ such that $Su=\lambda u$, $\lambda\in\C$,
we define 
\begin{align}
x^{S}u&=x^{\lambda}u.
\end{align}
We extend $x^{S}$ for an arbitrary $u\in M$ by linearity.
For $u\in M$ we define
\begin{align}
\label{eq:jordan-power}
x^{N}u&=e^{N\log x}u=\sum_{i=0}^{\infty}\dfrac{(\log x)^{i}N^i}{i!}u\mbox{\quad and}\nonumber\\
x^{L(0)}u&=x^{S}x^{N}u.
\end{align} 
We also define $x^{-L(0)}u$ by the same manner.
For $u\in M$ we clearly have
\begin{align}
\label{eq:derivation-xL0}
x\dfrac{d}{dx}x^{\pm L(0)}u&=x^{\pm L(0)}(\pm L(0))u.
\end{align} 
Although the following generalization of \eqref{eq:conj-L(0)-re}
seems to be well known,
we give a proof.
\begin{lemma}
\label{lemma:log-0}
Let $W_i,i=1,2,3$ be as above.
For a logarithmic intertwining operator $I(\ ,x) : W_1\otimes_{\C}W_2\rightarrow W_3[\log x]\{x\}$,
we have
\begin{align}
\label{eq:l0-int}
I(u,x)v&=\sum_{i\in\Z}x^{L(0)}I^{o}(x^{-L(0)}u;i)x^{-L(0)}v
\end{align}
for $u\in W_1$ and $v\in W_2$, where the actions of $x^{\pm L(0)}$ on $W_i,i=1,2,3$
are defined by \eqref{eq:jordan-power}.
\end{lemma}
\begin{proof}
For $\ewo\in W_1$ and $\ewt\in W_2$, we denote by $J(u,x)v$ the right-hand side of \eqref{eq:l0-int}.
By \eqref{eq:jordan-power}, $J(u,x)v$ can be written as
\begin{align}
\label{eq:inter-expand}
J(u,x)v
&=\sum_{i\in\Z}e^{N\log x}I^{o}(e^{-N\log x}u;i)e^{-N\log x}v x^{\lambda_3-\lambda_1-\lambda_2-i-1}\nonumber\\
&\in W_3[\log x]\{x\}.
\end{align}
and therefore $J(u,x)v$ satisfies \eqref{eq:truncation}. 

Since $N$ is a $V$-module homomorphism by \cite[Proposition 2.2]{HLZ}, $J(\ ,x)$ satisfies \eqref{eq:log-borchrerds}.
By \eqref{eq:derivation-xL0}, we have \eqref{eq:derivative-property} as follows:
\begin{align}
&x\dfrac{d}{dx}J(u,x)v\nonumber\\
&=\sum_{i\in\Z}\big(
x^{L(0)}L(0)I^{o}(x^{-L(0)}u;i)x^{-L(0)}v-x^{L(0)}I^{o}(x^{-L(0)}L(0)u;i)x^{-L(0)}v\nonumber\\
&\quad{}-x^{L(0)}I^{o}(x^{-L(0)}u;i)x^{-L(0)}L(0)v\big)\nonumber\\
&=\sum_{i\in\Z}x^{L(0)}I^{o}(x^{-L(0)}L(-1)u;i+1)x^{-L(0)}v.
\end{align}
Thus, $J(\ ,x)$ is a logarithmic intertwining operator.
If we wright
\begin{align}
J(\ ,x)&=\sum_{i=0}^{\infty}J^{(i)}(\ ,x)(\log x)^i,& J^{(i)}(\ ,x) : W_1\otimes_{\C}W_2\rightarrow W_3\{x\},
\end{align}
then we have $J^{(0)}(\ ,x)=I^{(0)}(\ ,x)$ by \eqref{eq:inter-expand}
and therefore $I(\ ,x)=J(\ ,x)$ by the comment right after
\eqref{eq:inductive-inter}.
\end{proof}

For a logarithmic intertwining operator $I(\ ,x) : W_1\otimes_{\C}W_2\rightarrow W_3[\log x]\{x\}$,
$I^{o}(\ ,x) : W_1\otimes_{\C}W_2\rightarrow W_3\db{x}$ is a $\Z$-graded intertwining operator
since $I^{o}(\ ,x)$ satisfies \eqref{eq:inter-borcherds}.
Conversely, the proof of Lemma \ref{lemma:log-0} shows that
for a $\Z$-graded intertwining operator $\Phi(\ ,x) : W_1\otimes_{\C}W_2\rightarrow W_3\db{x}$,
the map 
\begin{align}
W_1\otimes_{\C}W_2&\rightarrow W_3[\log x]\{x\}\nonumber\\
\ewo\otimes\ewt& \mapsto \sum_{i\in\Z}x^{L(0)}\Phi(x^{-L(0)}\ewo;i)x^{-L(0)}\ewt
\end{align}
is a logarithmic intertwining operator.
Thus we  have the following result.
\begin{proposition}
\label{relation-log-N}
Let $W_1, W_2$, and $W_3$ be three weak $V$-modules which satisfy
\eqref{eq:finite-module-1}. Then,
the map 
\begin{align}
\label{eq:one-to-one-log-N}
I_{\log}\binom{W_3}{W_1\ W_2}&\rightarrow I_{\Z}\binom{W_3}{W_1\ W_2}\nonumber\\
I(\ ,x)&\mapsto I^{o}(\ ,x)
\end{align}
is a linear isomorphism.
\end{proposition}

\section{\label{section:correspondence}The main theorem}

Throughout this section $\lwt$ and $\lws$ are two left $A(V)$-modules and  
$W_1$ is an $\N$-graded weak $V$-module.
In this section we establish a one-to-one correspondence between 
$\Hom_{A(V)}(A(W_1)\otimes_{A(V)} \lwt, \lws)$ and
$I_{\Z}\binom{S(\lws^{*})^{\prime}}{W_1\ S(\lwt)}$
as a generalization of \cite[Theorem 1.5.3]{FZ} and \cite[Theorem 2.11]{Li}.
Here for a left $A(V)$-module $U$, 
$S(U)=\oplus_{j=0}^{\infty}S(U)(j)$ is the generalized Verma $V$-module
with $S(U)(0)=U$ 
defined in Section \ref{section:preliminary}, $U^{*}=\Hom_{\C}(U,\C)$, and $S(U)^{\prime}=\oplus_{j=0}^{\infty}\Hom_{\C}(S(U)(j),\C)$.
We will show this result by modifying the proofs of
\cite[Theorem 6.6]{HY}, \cite[Theorem 2.11]{Li}, and \cite[Theorem 2.2.1]{Z} so as not to  use the $L(-1)$- derivative property.

For a vector space $U$, $T(U)$ denotes the tensor algebra of $U$.
For an $\N$-graded weak $V$-module $W_1$,
$T(W_1,\lwt)$ denotes the tensor algebra $T((V\oplus W_1\oplus \lwt)[t,t^{-1}])$ 
and $\free(W_1,\lwt)$ denotes
the subspace $T(V[t,t^{-1}])\otimes_{\C} W_1[t,t^{-1}]\otimes_{\C}T(V[t,t^{-1}])\otimes_{\C}\lwt$ of 
$T(W_1,\lwt)$.
For simplicity we shall omit the tensor product symbol.
For $a\in V\oplus W_1\oplus \lwt$ and $i\in\Z$, $a(i)$ denotes $a\otimes t^{i}$.
For $a\in V\oplus W_1\oplus \lwt$, we define a map
\begin{align}
Y_{T(W_1,\lwt)}(a,x) : T(W_1,\lwt)& \rightarrow T(W_1,\lwt)\nonumber\\
u& \mapsto \sum_{i\in\Z}a(i)ux^{-i-1}.
\end{align}
We note that for $a\in V$
\begin{align}
Y_{T(W_1,\lwt)}(a,x)(T(V[t,t^{-1}])\otimes_{\C}\lwt)&\subset  (T(V[t,t^{-1}])\otimes_{\C}\lwt)[[x,x^{-1}]],\nonumber\\
Y_{T(W_1,\lwt)}(a,x)(F(W_1,\lwt))&\subset F(W_1,\lwt)[[x,x^{-1}]]
\end{align}
and for $u\in W_1$
\begin{align}
Y_{T(W_1,\lwt)}(u,x)(T(V[t,t^{-1}])\otimes_{\C}\lwt)&\subset  F(W_1,\lwt)[[x,x^{-1}]].
\end{align}
For homogeneous
 $a^1,\ldots,a^{n-2}\in V$,
homogeneous $\ewo\in W_1$,  $\ewt\in \lwt$,  
$m_1,\ldots,m_{n-2},i\in\Z$, and $s\in\{1,\ldots,n-2\}$,
we define the {\em degree} of 
\begin{align}
\label{eqn:element}
H=a^1(m_1)\cdots a^{s}(m_s)u(i)a^{s+1}(m_{s+1})\cdots a^{n-2}(m_{n-2})v\in F(W_1,\lwt)
\end{align}
by
\begin{align}
\label{eqn:deg-element}
\deg H&=\sum_{j=1}^{n-2}(\wt a_{j}-m_{j}-1)+(\deg u-i-1).
\end{align}
For $n\in\Z$, we denote by $\free(W_1,\lwt)(n)$ the set of all elements in $\free(W_1,\lwt)$ with degree $n$.
Then, we have
\begin{align}
\free(W_1,\lwt)&=\bigoplus_{n\in\Z}\free(W_1,\lwt)(n).
\end{align}

\begin{definition}
\label{definition:ideal}
Let $\ideal(W_1,\lwt)$ be the subspace of $\free(W_1,\lwt)$ generated by the following elements:
\begin{enumerate}
\item
all elements in $\oplus_{n<0}\free(W_1,\lwt)(n)$.
\item For homogeneous $a,b\in V$, $\wpp\in T(V[t,t^{-1}])$,
homogeneous $\wq\in F(W_1,\lwt)$, and $l,m,n\in\Z$,
\begin{align}
\label{eq:borcherds-pre}
&\wpp\Big(\sum\limits_{i=0}^{\infty}
\binom{m}{i}(a_{l+i}b){(m+n-i)}\nonumber\\
& -\sum\limits_{i=0}^{\wt b-n-1+\deg \wq}
(-1)^i\binom{l}{i}
a{(m+l-i)}b{(n+i)}\nonumber\\
&\quad{}-(-1)^{l+1}\sum\limits_{i=0}^{\wt a-m-1+\deg \wq}(-1)^i\binom{l}{i}b{(n+l-i)}a{(m+i)}\Big)\wq.
\end{align}
\item
For $\wpp\in T(V[t,t^{-1}])$,
$\wq\in F(W_1,\lwt)$, and $n\in\Z$,
\begin{align}
\label{eqn:1(n)-}
\wpp(\1(n)-\delta_{n,-1})\wq.
\end{align}
\item
For homogeneous $a\in V$, $v\in \lwt$, and $\wpp\in T(V[t,t^{-1}])\otimes_{\C}W_1[t,t^{-1}]\otimes_{\C}T(V[t,t^{-1}])$,
\begin{align}
\wpp(a(\wt a-1)-o(a))v.
\end{align}
\item
For homogeneous $a,b\in V$, homogeneous $q\in T(V[t,t^{-1}])\otimes \lwt$, $\wpp\in T(V[t,t^{-1}])\otimes_{\C}W_1[t,t^{-1}]\otimes_{\C}T(V[t,t^{-1}])$, and $l,m,n\in\Z$,
\begin{align}
&\wpp\Big(\sum\limits_{i=0}^{\infty}
\binom{m}{i}(a_{l+i}b){(m+n-i)}\nonumber\\
& -\sum\limits_{i=0}^{\wt b-n-1+\deg q}
(-1)^i\binom{l}{i}
a{(m+l-i)}b{(n+i)}\nonumber\\
&\quad{}-(-1)^{l+1}\sum\limits_{i=0}^{\wt a-m-1+\deg q}(-1)^i\binom{l}{i}b{(n+l-i)}a{(m+i)}\Big)\wq.
\end{align}
\item
For $\wpp\in T(V[t,t^{-1}])\otimes_{\C}W_1[t,t^{-1}]\otimes_{\C}T(V[t,t^{-1}])$,
$\wq\in T(V[t,t^{-1}])\otimes \lwt$, and $n\in\Z$,
\begin{align}
\wpp(\1(n)-\delta_{n,-1})\wq.
\end{align}
\item
For homogeneous $a\in V$, homogeneous $u\in W_1$, homogeneous $\wq\in T(V[t,t^{-1}])\otimes \lwt$,
$\wpp\in T(V[t,t^{-1}])$, and $l,m,n\in\Z$,
\begin{align}
&\wpp\Big(\sum\limits_{i=0}^{\infty}
\binom{m}{i}(a_{l+i}\ewo){(m+n-i)}\nonumber\\
& -\sum\limits_{i=0}^{\deg u-n-1+\deg q}
(-1)^i\binom{l}{i}
a{(m+l-i)}\ewo{(n+i)}\nonumber\\
&\quad{}-(-1)^{l+1}\sum\limits_{i=0}^{\wt a-m-1+\deg q}(-1)^i\binom{l}{i}\ewo{(n+l-i)}a{(m+i)}\Big)q.
\end{align}
\end{enumerate}
\end{definition}
Since $\ideal(W_1,\lwt)$ is generated by homogeneous elements, we have
\begin{align}
\ideal(W_1,\lwt)&=\bigoplus_{n\in\Z}(\ideal(W_1,\lwt)\cap \free(W_1,\lwt)(n)).
\end{align}
We set
\begin{align}
\label{eq:definition-S}
S(W_1,\lwt)&=\free(W_1,\lwt)/\ideal (W_1,\lwt)
\end{align}
and 
\begin{align}
S(W_1,\lwt)(n)&=\free(W_1,\lwt)(n)+\ideal (W_1,\lwt)
\end{align}
for $n\in\Z$. We have $S(W_1,\lwt)(n)=0$ for $n<0$ by Definition \ref{definition:ideal} (1) and
\begin{align}
S(W_1,\lwt)&=\bigoplus_{n=0}^{\infty}S(W_1,\lwt)(n).
\end{align}
We shall use elements of $T(W_1,\lwt)$ to represent elements of $S(W_1,\lwt)$.
For $a\in V$, $Y_{S(W_1,\lwt)}(a,x)$ denotes
the map $S(W_1,\lwt)\rightarrow S(W_1,\lwt)\db{x}$ induced by $Y_{T(W_1,\lwt)}(\ ,x)$, namely
\begin{align}
Y_{S(W_1,\lwt)}(a,x) : S(W_1,\lwt)&\rightarrow S(W_1,\lwt)\db{x}\nonumber\\
\ewo&\mapsto \sum_{i\in\Z}a(i)\ewo x^{-i-1}.
\end{align}
By definition, $S(W_1,\lwt)$ is an $\N$-graded weak $V$-module.

For a vector space $U$, we define
\begin{align}
\label{eq:Uwx-wx}
{U}_{\{\wy_1,\ldots,\wy_n\}}&=U[[\wy_i-\wy_j\ |\ 1\leq i<j\leq n]][(\wy_i-\wy_j)^{-1}\ |\ 1\leq i<j\leq n].
\end{align}
For distinct $i,j\in\{1,\ldots,n\}$, let
\begin{align}
\label{eq:iotaij}
\iota_{(i,j)} : {U}_{\{\wy_1,\ldots,\wy_n\}}&\rightarrow 
{U}_{\{\wy_1,\ldots,\widehat{\wy_{i}},\ldots,\wy_n\}}\db{\wy_{i}-\wy_{j}}
\end{align}
be a linear map, where $\widehat{y_i}$ denotes the omission of the term $y_i$,
defined by $\iota_{(i,j)}(u)=u$ for $u\in U$ and
\begin{align}
\iota_{(i,j)} (\wy_{k}-\wy_{l})^{m}&
=
\left\{
\begin{array}{ll}
(\wy_{k}-\wk_{l})^{m},&\mbox{if }k,l\neq i,\\
\sum_{s=0}\limits^{\infty}\dbinom{m}{s}(\wy_{j}-\wy_{l})^{m-s}(\wy_{i}-\wy_{j})^{s}
,&\mbox{if }k=i, \mbox{ and }\\
\sum_{s=0}\limits^{\infty}\dbinom{m}{s}(\wy_{k}-\wy_{j})^{m-s}(-\wy_{i}+\wy_{j})^{s}
,&\mbox{if }l=i
\end{array}
\right.
\end{align} 
for distinct $k,l\in\{1,\ldots,n\}$ and $m\in\Z$. 
For $i_1,j_1,i_2,j_2\in\{1,\ldots,n\}$ such that $i_1\neq j_1$, $i_2\neq j_2$, and $j_1\neq i_2$,
we  define a map 
\begin{align}
\iota_{(i_1,j_1),(i_2,j_2)} : &\  U_{\{y_1,\ldots,y_n\}}\rightarrow
 U_{\{y_1,\ldots,\widehat{y_{i_1}},\ldots,\widehat{y_{i_2}},\ldots,y_n\}}\db{y_{i_1}-y_{j_1}}\db{y_{i_2}-y_{j_2}}
\end{align}
as follows: for $f\in U_{\{y_1,\ldots,y_n\}}$, writing
\begin{align}
\iota_{(i_2,j_2)}f&
=\sum_{k\in\Z}f_{k}(y_{i_2}-y_{j_2})^{k},
\ f_k\in U_{\{y_1,\ldots,\widehat{y_{i_2}},\ldots,y_n\}}\nonumber\\
&\in 
 U_{\{y_1,\ldots,\widehat{y_{i_2}},\ldots,y_n\}}\db{y_{i_2}-y_{j_2}},
\end{align}
we define 
\begin{align}
\label{eqn:iota-i-2-j-2}
\iota_{(i_1,j_1),(i_2,j_2)}f&=\sum_{k\in\Z}(\iota_{(i_1,j_1)}f_{k})(y_{i_2}-y_{j_2})^{k}\nonumber\\
&\in U_{\{y_1,\ldots,\widehat{y_{i_1}},\ldots,\widehat{y_{i_2}},\ldots,y_n\}}\db{y_{i_1}-y_{j_1}}\db{y_{i_2}-y_{j_2}}.
\end{align}
By the same manner we inductively define a map
\begin{align}
\label{eqn:iota-i-k-j-k}
\iota_{(i_1,j_1),\ldots,(i_k,j_k)} : &\  U_{\{y_1,\ldots,y_n\}}\rightarrow
 U_{\{y_1,\ldots,\widehat{y_{i_1}},\ldots,\widehat{y_{i_k}},\ldots,y_n\}}\db{y_{i_1}-y_{j_1}}\cdots \db{y_{i_k}-y_{j_k}}
\end{align}
for $i_1,j_1,\ldots,i_k,j_k\in \{1,\ldots,n\}$ such that
$i_m\neq j_m$ and $j_{m}\not\in\{i_{m+1},\ldots,i_{k}\}$ for all $m=1,\ldots,k$.
We note that for distinct $i_1,\ldots,i_n\in\{1,\ldots,n\}$
 $\iota_{(i_1,j_1),(i_2,j_2),\ldots,(i_{n},j_{n})}$ and 
$\iota_{(i_2,j_2),\ldots,(i_{n},j_{n})}$ are the same maps on $U_{\{y_1,\ldots,y_n\}}$ by definition.

Let $U$ be a vector space over $\C$ and $h\in\Z$.
We say $p\in U[[(y_i-y_j)^{\pm 1}\ |\ 1\leq i<j\leq n]]$
is {\it homogeneous of total degree} $h$ 
if all the terms appearing
in it with nonzero coefficients have the same total degree $h$.
We note that for every distinct $i,j\in\{1,\ldots,n\}$,
$p\in U_{\{y_1,\ldots,y_n\}}$ is homogeneous of total degree $h$
if and only if so is $\iota_{(i,j)}p$.

We write 
\begin{align}
\label{eq:lw-v}
V&=\bigoplus_{i=\Delta}^{\infty}V_i
\end{align}
where $V_i=\{a\in V\ |\ L(0)a=ia\}$.
Let $M$ be a weak $V$-module. For homogeneous $a^1,\ldots,a^{n-1}\in V$ and homogeneous $u\in M$,
standard arguments (cf. \cite[Sections 3.2--3.4]{LL} and \cite[Lemma 2.4]{T2}) show that 
there exists 
\begin{align}
\label{eq:hatY-bound}
&\hat{Y}_{M}(a^1,\ldots,a^{n-1},u|\wy_1,\ldots,\wy_n)\nonumber\\
&\in 
\prod_{1\leq s<t\leq n-1}(y_s-y_t)^{-\wt a^s-\wt a^t+\Delta}
\prod_{1\leq m< n}(y_m-y_n)^{-\wt a^m-\deg u}\nonumber\\
&\quad{}\times M[[y_i-y_j\ |\ 1\leq i<j\leq n]]\nonumber\\
&\subset {M}_{\{\wy_1,\ldots,\wy_n\}}
\end{align}
such that
\begin{align}
\label{eq:a-expansion}
&\iota_{(1,n),(2,n)\ldots,(n-1,n)}\hat{Y}_{M}(a^1,a^2,\ldots,a^{n-1},u|\wy_1,\ldots,\wy_n)\nonumber\\
&=Y_{M}(a^1,y_1-y_n)Y_{M}(a^2,y_2-y_n)\cdots Y_{M}(a^{n-1},y_{n-1}-y_n)u\nonumber\\
&\in M\db{y_1-y_n}\db{y_2-y_n}\cdots\db{y_{n-1}-y_n}.
\end{align}
We note that
\begin{align}
\hat{Y}_{M}(a,u|\wy_1,\wy_2)&=Y_{M}(a,y_1-y_2)u
\end{align}
for $a\in V$ and $u\in M$. Standard arguments (cf. \cite[Sections 3.2--3.4]{LL} and \cite[Lemma 2.4]{T2}) also show that
\begin{align}
\label{eq:a-permutation}
&\hat{Y}_{M}(a^{\sigma(1)},\ldots,a^{\sigma(n-1)},u|\wy_{\sigma(1)},\ldots,\wy_{\sigma(n-1)},\wy_n)\nonumber\\
&=\hat{Y}_{M}(a^1,\ldots,a^{n-1},u|\wy_1,\ldots,\wy_n)
\end{align}
for an arbitrary permutation $\sigma$ of $\{1,\ldots,n-1\}$,
\begin{align}
&\iota_{(i,i+1)}\hat{Y}_{M}(a^1,\ldots,a^{n-1},u|\wy_1,\ldots,\wy_n)\nonumber\\
&=\hat{Y}_{M}(a^1,\ldots,a^{i-1},Y(a^{i},\wy_{i}-\wy_{i+1})a^{i+1},a^{i+2},\ldots,a^{n-1},u|\wy_1,
\ldots,\widehat{\wy_{i}},\ldots,\wy_n)\nonumber\\
&\in 
M_{\{y_1,\ldots,\widehat{\wy_{i}},\ldots,y_{n}\}}\db{\wy_i-\wy_{i+1}}
\end{align}
for $i=1,\ldots,n-2$, and
\begin{align}
&\iota_{(n-1,n)}\hat{Y}_{M}(a^1,\ldots,a^{n-1},u|\wy_1,\ldots,\wy_n)\nonumber\\
&=\hat{Y}_{M}(a^1,\ldots,a^{n-2},Y_{M}(a^{n-1},\wy_{n-1}-\wy_{n})u|\wy_1,
\ldots,\wy_{n-2},\wy_n)\nonumber\\
&\in 
M_{\{y_1,\ldots,y_{n-2},y_{n}\}}
\db{\wy_{n-1}-\wy_{n}}.
\end{align}

For $n\in\N$, $V^{\times n}$ denotes the $n$ times direct product of $V$.
Let 
\begin{align}
f_{n} : & V^{\times n-2}\times W_1\times \lwt
\rightarrow U_{\{y_1,\ldots,y_{n}\}},\ n=2,3,\ldots
\end{align}
be a sequence of maps which satisfies the following conditions:
let $a^1,\ldots,a^{n-2}\in V$, $u\in W_1$, and $v\in\lwt$.
\begin{enumerate}
\item
For an arbitrary permutation $\sigma$ of $\{1,\ldots,n-2\}$,
\begin{align}
\label{eq:permutation}
&f_{n}(a^1,\ldots,a^{n-2},u,v| y_1,\ldots,\wy_{n-2},\wy_{n-1},\wy_{n})\nonumber\\
&=
f_{n}(a^{\sigma(1)},\ldots,a^{\sigma(n-2)},u,v| y_{\sigma(1)},\ldots,
\wy_{\sigma(n-2)},\wy_{n-1},\wy_{n}).
\end{align}
\item For $i=1,\ldots,n-3$, 
\begin{align}
&\iota_{(i,i+1)}f_{n}(a^1,\ldots,a^{n-2},u,v| y_1,\ldots,\wy_{n})\nonumber\\
&=f_{n-1}(a^1,\ldots,Y(a^{i},\wy_i-\wy_{i+1})a^{i+1},\ldots,a^{n-2},u,v| y_1,\ldots,\widehat{\wy_{i}},
\ldots,\wy_{n})\nonumber\\
&\in {U}_{\{\wy_1,\ldots,\widehat{\wy_{i}},\ldots,\wy_{n}\}}\db{\wy_{i}-\wy_{i+1}}.
\end{align}
\item 
\begin{align}
&\iota_{(n-2,n-1)}f_{n}(a^1,\ldots,a^{n-3},a^{n-2},u,v| y_1,\ldots,\wy_{n})\nonumber\\
&=f_{n-1}(a^1,\ldots,a^{n-3},Y_{W_1}(a^{n-2},\wy_{n-2}-\wy_{n-1})u,v| y_1,\ldots,\wy_{n-3},
\wy_{n-1},\wy_{n})\nonumber\\
&\in {U}_{\{\wy_1,\ldots,\wy_{n-3},\wy_{n-1},\wy_{n}\}}\db{\wy_{n-2}-\wy_{n-1}}.
\end{align}
\item 
If $a^{n-2}$ is homogeneous, then the coefficient of $(\wy_{n-2}-\wy_{n})^{-\wt a^{n-2}}$ in
\begin{align*}
&\iota_{(n-2,n)}f_{n}(a^1,\ldots,a^{n-3},a^{n-2},u,v| y_1,\ldots,\wy_{n})
\end{align*}
is equal to 
\begin{align}
\label{eq:a-v}
f_{n-1}(a^1,\ldots,a^{n-3},u,o(a^{n-2})v| y_1,\ldots,
\wy_{n-3},\wy_{n-1},\wy_{n}).
\end{align}
\end{enumerate}
We define a map $\Phi : F(W_1,\lwt)\rightarrow U$ by
\begin{align}
\label{eq:definition-expand}
&\Phi(Y_{T(W_1,\lwt)}(a^1,\wy_1-\wy_{n})\cdots
Y_{T(W_1,\lwt)}(a^s,\wy_{s}-\wy_{n})\nonumber\\
&\qquad{}\times Y_{T(W_1,\lwt)}(u,\wy_{n-1}-\wy_{n})\nonumber\\
&\qquad{}\times Y_{T(W_1,\lwt)}(a^{s+1},\wy_{s+1}-\wy_{n})\cdots Y_{T(W_1,\lwt)}(a^{n-2},\wy_{n-2}-\wy_{n})
v)\nonumber\\
&=\iota_{(1,n),\ldots,(s,n),(n-1,n),(s+1,n),\ldots,(n-2,n)}
f_n(a^1,\ldots,a^{n-2},u,v| y_1,\ldots,y_{n})
\end{align} 
for $a^1,\ldots,a^{n-2}\in V, \ewo\in W_1$, $\ewt\in \lwt$, and $s=1,\ldots,n-2$.
\begin{lemma}
\label{lemma:rational-function}
With the notation above,
$\Phi(J(W_1,\lwt))=0$ and therefore
 the map $\Phi : T(W_1,\lwt) \rightarrow U$ 
induces a map $S(W_1,\lwt) \rightarrow U$ which denoted by the same symbol:
\begin{align}
\Phi:\quad S(W_1,\lwt)& \rightarrow U\nonumber\\
u& \mapsto \Phi(u).
\end{align}
\end{lemma}
\begin{proof}
We simply write $Y=Y_{T(W_1,\lwt)}$.
We only show that the images of elements of the forms \eqref{eq:borcherds-pre} and \eqref{eqn:1(n)-} in Definition \ref{definition:ideal}
vanish. We can show the images of the other elements in Definition \ref{definition:ideal} vanish in the same manner.

For $i,s=1,\ldots,n-3$ with $i+1<s$, defining
\begin{align}
P&=\Yu (a^1,y_1-y_n)\cdots \Yu (a^{i-1},y_{i-1}-y_{n}),\nonumber\\
Q&=
\Yu (a^{i+2},y_{i+2}-y_n)\cdots 
\Yu (a^{s},y_{s}-y_n) \Yu (u,y_{n-1}-y_n)\nonumber\\
&\quad{}\times \Yu (a^{s+1},y_{s+1}-y_n)\cdots \Yu (a^{n-2},y_{n-2}-y_{n})v,\nonumber\\
\iota_1&=\iota_{(1,n),\ldots,(i-1,n)},\mbox{ and }\nonumber\\
\iota_2&=\iota_{(i+2,n),\ldots,(s,n),(n-1,n),(s+1,n),\ldots,(n-2,n)},
\end{align}
we have
\begin{align}
\label{eq:3-composition}
&\iota_1\circ \iota_{(i,n),(i+1,n)}\circ \iota_2(f_{n}(a^1,\ldots,a^{n-2},u,v| y_1,\ldots,y_{n}))\nonumber\\
&=f_{2}(P \Yu (a^i,y_i-y_{n})\Yu (a^{i+1},y_{i+1}-y_{n})Q),\nonumber\\
&\iota_1\circ \iota_{(i+1,n),(i,n)}\circ \iota_2(f_n(a^1,\ldots,a^{n-2},u,v| y_1,\ldots,y_{n})\nonumber\\
&=f_{2}(P\Yu (a^{i+1},y_{i+1}-y_{n})\Yu (a^i,y_i-y_{n})Q),\quad\mbox{and}\nonumber\\
&\iota_1\circ \iota_{(i+1,n),(i,i+1)}\circ \iota_2(f_n(a^1,\ldots,a^{n-2},u,v| y_1,\ldots,y_{n})\nonumber\\
&=f_{2}(P\Yu (Y(a^i,y_i-y_{i+1})a^{i+1},y_{i+1}-y_{n})Q).
\end{align}
Let $p$ be an arbitrary coefficient of $P$ and 
$q$ an arbitrary coefficient of $Q$. 
Taking $a^i=a$ and $a^{i+1}=b$ in \eqref{eq:3-composition},
we have \eqref{eq:borcherds-pre} by 
standard arguments (cf. \cite[Sections 3.2--3.4]{LL} and \cite[Lemma 2.4]{T2}).

For $i=1,\ldots,n-3$, we have
\begin{align}
\label{eq:unit-v}
&\iota_{(i,i+1)}(f_{n}(\ldots,a^{i-1},\1,a^{i+1},\ldots | y_1,\ldots,y_{n}))\nonumber\\
&=f_{n-1}(\ldots,a^{i-1},\Yu(\Yu (\1,y_i-y_{i+1})a^{i+1},\ldots | y_1,\ldots,\widehat{\wy_{i}},\ldots,y_n)\nonumber\\
&=f_{n-1}(\ldots,a^{i-1},a^{i+1},\ldots | y_1,\ldots,\widehat{\wy_{i}},\ldots,y_n)\nonumber\\
&\in T(W_1,\lwt)_{\{y_1,\ldots,\widehat{\wy_{i}},\ldots,y_{n}\}}.
\end{align}
Since $T(W_1,\lwt)_{\{y_1,\ldots,\widehat{\wy_{i}},\ldots,y_{n}\}}
\subset T(W_1,\lwt)_{\{y_1,\ldots,y_{n}\}}$ and $\iota_{(i,i+1)}$ is injective, we have
\begin{align}
\label{eq:(n-1)-n}
&f_{n}(\ldots,a^{i-1},\1,a^{i+1},\ldots | y_1,\ldots,y_{n})\nonumber\\
&=f_{n-1}(\ldots,a^{i-1},a^{i+1},\ldots | y_1,\ldots,\widehat{\wy_{i}},\ldots,y_n)
\end{align}
and therefore \eqref{eqn:1(n)-} by \eqref{eq:definition-expand}.
\end{proof}

For $\Phi(\ ,x)\in \inter\binom{W_3}{W_1\ W_2}$ and homogeneous $u\in W_1$,
we denote $\Phi(u;\deg u-1)$ by $o^{\Phi}(u)$
and extend $o^{\Phi}(u)$ for an arbitrary $u\in W_1$ by linearity. 
The map 
\begin{align}
o^{\Phi} : W_1\otimes W_2(0)&\rightarrow W_3(0)\nonumber\\
u\otimes v&\mapsto o^{\Phi}(u)v
\end{align}
induces an $A(V)$-module homomorphism $A(W_1)\otimes_{A(V)}W_2(0)\rightarrow W_3(0)$ which denoted 
by the same symbol:
\begin{align}
\label{eq:inter-zhu-2}
o^{\Phi} : A(W_1)\otimes_{A(V)}W_2(0)&\rightarrow W_3(0) \nonumber\\
u\otimes v&\mapsto o^{\Phi}(u)v.
\end{align}
We get a map
\begin{align}
\label{eq:inter-zhu}
\inter\binom{W_3}{W_1\ W_2}&\rightarrow \Hom_{A(V)}(A(W_1)\otimes_{A(V)}W_2(0),W_3(0))\nonumber\\
\Phi(\ ,x)&\mapsto o^{\Phi}.
\end{align}

The following is the main result.
\begin{theorem}\label{theorem:correspondence}
For  an $\N$-graded weak $V$-module $W_1$ and
two left $A(V)$-modules $\Omega_{(2)}$ and $\Omega_{(3)}$,
the  map
\begin{align}
\label{eq:one-one-iinter-N}
\inter\binom{S(\lws^{*})^{\prime}}{W_1\ S(\lwt)}
&\rightarrow \Hom_{A(V)}(A(W_1)\otimes_{A(V)}\lwt,\lws)\nonumber\\
\Phi(\ ,x)&\mapsto o^{\Phi}
\end{align}
is a linear isomorphism.
\end{theorem}
\begin{proof}
The same argument as in the proof of \cite[Proposition 2.10]{Li}
shows the map \eqref{eq:one-one-iinter-N} is injective.
Let $\varphi : A(W_1)\otimes_{A(V)}\lwt\rightarrow \lws$ be an $A(V)$-module homomorphism.
For homogeneous $u\in W_1$ and $v\in \lwt$, we  define 
\begin{align}
\label{eq:varphi-first}
\varphi(u,v | \wy_1,\wy_2)
&=\varphi(u\otimes v)(\wy_1-\wy_2)^{-\deg u}
\end{align}
and extend $\varphi(u,v | \wy_1,\wy_2)$ for an arbitrary $u\in W_1$ by linearity.
For $a^1,\ldots,a^{n-2}\in V$, $\ewo\in W_1$, $\ewt\in \lwt$, 
and $i,j\in\{1,\ldots,n\}$ with $i<j$,
as temporary notation let us put
\begin{align}
\label{eq:Gamma-a}
\Gamma_{ij}&=
\left\{
\begin{array}{ll}
\wt a^i+\wt a^j-\Delta,&\mbox{ if }1\leq i<j\leq n-2,\\
\wt a^i+\deg u,&\mbox{ if }1\leq i\leq n-2\mbox{ and }j=n-1,\\
\wt a^i,&\mbox{ if }1\leq i\leq n-2\mbox{ and }j=n,\\
\deg u,&\mbox{ if }i=n-1 \mbox{ and }j=n
\end{array}
\right.
\end{align}
and
\begin{align}
\label{eq:definition-g}
\tf
&=
\prod_{1\leq i<j\leq n}(y_i-y_j)^{\Gamma_{ij}}
\varphi(\hY_{W_1}(a^1,\ldots,a^{n-2},u | \wy_1,\ldots,\wy_{n-2},\wy_{n-1}),v | \wy_{n-1},\wy_{n})
\end{align}
where $\Delta$ is defined in \eqref{eq:lw-v}.
Since 
\begin{align}
\label{eq:be-power-series}
&\prod_{1\leq i<j\leq n-1}(y_i-y_j)^{\Gamma_{ij}}
\hY_{W_1}(a^1,\ldots,a^{n-2},u | \wy_1,\ldots,\wy_{n-2},\wy_{n-1})\nonumber\\
&\in W_1[[\wy_i-\wy_j\ |\ 1\leq i<j\leq n-1]]
\end{align}
by \eqref{eq:hatY-bound}, we have
\begin{align}
\label{eq:g-in}
\tf
&\in (\lws\db{\wy_{n-1}-\wy_{n}})[[\wy_i-\wy_j\ |\ 1\leq i<j\leq n-1]]\nonumber\\
&\qquad{}\times[(y_1-y_{n})^{\pm 1},\ldots,(y_{n-2}-y_{n})^{\pm 1}].
\end{align}
We note that $\tf$ is homogeneous of total degree 
\begin{align}
\label{eq:total-degree-g}
\Gamma&=\sum_{1\leq i<j\leq n}\Gamma_{ij}-\sum_{i=1}^{n-2}\wt a^i-\deg u.
\end{align}
For an arbitrary permutation $\sigma$ of $\{1,\ldots,n-2\}$,
by \eqref{eq:be-power-series}
and 
\begin{align}
&\iota_{(1,n-1),\ldots,(n-2,n-1)}
\prod_{i=1}^{n-1}(y_i-y_n)^{\Gamma_{in}}\nonumber\\
&=
\iota_{(\sigma(1),n-1),\ldots,(\sigma(n-2),n-1)}
\prod_{i=1}^{n-1}(y_i-y_n)^{\Gamma_{in}},
\end{align}
we  have
\begin{align}
\iota_{(1,n-1),\ldots,(n-2,n-1)}g
&=
\iota_{(\sigma(1),n-1),\ldots,(\sigma(n-2),n-1)}g
\end{align}
 and therefore
\begin{align}
\label{eq:1-n-2-permutation}
&\iota_{(1,n-1),\ldots,(n-2,n-1)}
\Big(\prod_{1\leq i<j\leq n}(y_i-y_j)^{\Gamma_{ij}}\nonumber\\
&\quad{}\times
\varphi(
Y_{W_1}(a^1,y_1-y_{n-1})\cdots Y_{W_1}(a^{n-2},y_{n-2}-y_{n-1})u,v|y_{n-1},y_{n})\Big)\nonumber\\
&=
\iota_{(\sigma(1),n-1),\ldots,(\sigma(n-2),n-1)}
\Big(\prod_{1\leq i<j\leq n}(y_i-y_j)^{\Gamma_{ij}}\nonumber\\
&\quad{}\times\varphi(
Y_{W_1}(a^{\sigma(1)},y_{\sigma(1)}-y_{n-1})\cdots Y_{W_1}(a^{\sigma(n-2)},y_{\sigma(n-2)}-y_{n-1})u,v|
y_{n-1},y_{n})\Big)
\end{align}
by \eqref{eq:a-expansion} and \eqref{eq:a-permutation}.

Let $i_1,\ldots,i_{n-2},i_{n}$ be a sequence of integers such that
\begin{align}
\label{eq:res-non-0}
0\neq &\Res_{y_{n-1}-y_{n}}\Res_{y_{1}-y_{n-1}}\cdots
\Res_{y_{n-2}-y_{n-1}}\nonumber\\
&\quad{}\times(y_{n-1}-y_{n})^{i_n}(y_1-y_{n-1})^{i_1}\cdots (y_{n-2}-y_{n-1})^{i_{n-1}}\nonumber\\
&\quad{}\times\iota_{(1,n-1),\ldots,(n-2,n-1)}g.
\end{align}
By \eqref{eq:total-degree-g}, we have
\begin{align}
\label{eq:total-degree-residue-g}
i_1+\cdots+i_{n-2}+i_{n}=-\Gamma-n+1.
\end{align}
By \eqref{eq:1-n-2-permutation},
we may assume that $i_1$ is the smallest element in $\{i_1,\ldots,i_{n-2}\}$.
Since
\begin{align}
&\iota_{(1,n-1),\ldots,(n-2,n-1)}
\prod_{1\leq \wi<{\wj}\leq n}(y_{\wi}-y_{\wj})^{\Gamma_{{\wi}{\wj}}}\nonumber\\
&=\sum_{k_{12},k_{13},\ldots,k_{n-3,n-2}=0}^{\infty}\sum_{k_{1n},\ldots,k_{n-2,n}=0}^{\infty}\Big(\prod_{1\leq {\wi}<{\wj}\leq n-2}
\binom{\Gamma_{{\wi}{\wj}}}{k_{{\wi}{\wj}}}\Big)\Big(\prod_{{\wi}=1}^{n-2}
\binom{\Gamma_{{\wi}n}}{k_{{\wi}n}}\Big)\nonumber\\
&\quad{}\times (y_{n-1}-y_{n})^{\sum_{m=1}^{n-1}\Gamma_{mn}-
\sum_{m=1}^{n-2}k_{mn}}\nonumber\\
&\quad{}\times \prod_{j=1}^{n-2}(-1)^{\sum_{m=1}^{j-1}k_{mj}}(y_j-y_{n-1})^{k_{jn}+\sum_{m=1}^{j-1}k_{mj}+
\sum_{m=j+1}^{n-1}\Gamma_{jm}-\sum_{m=j+1}^{n-2}k_{jm}}\nonumber\\
&=
\sum_{k_{12},k_{13},\ldots,k_{n-3,n-2}=0}^{\infty}\sum_{k_{2n},\ldots,k_{n-2,n}=0}^{\infty}\Big(\prod_{1\leq {\wi}<{\wj}\leq n-2}
\binom{\Gamma_{{\wi}{\wj}}}{k_{{\wi}{\wj}}}\Big)\Big(\prod_{{\wi}=2}^{n-2}
\binom{\Gamma_{{\wi}n}}{k_{{\wi}n}}\Big)\nonumber\\
&\quad{}\times \sum_{k_{1n}=0}^{\infty}\binom{\Gamma_{1n}}{k_{1n}}
(y_{n-1}-y_{n})^{\Gamma_{1n}-k_{1n}+\sum_{m=2}^{n-1}\Gamma_{mn}-\sum_{m=2}^{n-2}k_{mn}}\nonumber\\
&\quad{}\times 
(y_1-y_{n-1})^{k_{1n}+\sum_{m=2}^{n-1}\Gamma_{1m}-\sum_{m=2}^{n-2}k_{1m}}\nonumber\\
&\quad{}\times \prod_{j=2}^{n-2}(-1)^{\sum_{m=1}^{j-1}k_{mj}}(y_j-y_{n-1})^{k_{jn}+\sum_{m=1}^{j-1}k_{mj}+
\sum_{m=j+1}^{n-1}\Gamma_{jm}-\sum_{m=j+1}^{n-2}k_{jm}},
\end{align}
the right-hand side of \eqref{eq:res-non-0} can be written as
a linear combination of the following elements:
\begin{align}
\label{eq:deform-a-1}
&\Res_{y_{n-1}-y_{n}}\Res_{y_{1}-y_{n-1}}\sum_{k_{1n}=0}^{\infty}\binom{\Gamma_{1n}}{k_{1n}}
(y_{n-1}-y_{n})^{\Gamma_{1n}-k_{1n}+d}\nonumber\\
&\quad{}\times 
(y_1-y_{n-1})^{i_1+\sum_{\wm=2}^{n-1}\Gamma_{1\wm}-l+k_{1n}}\varphi(Y(a^1,y_1-y_{n-1})w,v|y_{n-1},y_{n})
\end{align}
where $l\in\N$, $d\in\Z$, and $w$ is a homogeneous element of $W_1$.
We see that \eqref{eq:deform-a-1} becomes
\begin{align}
\label{eq:deform-a-2}
&\Res_{y_{n-1}-y_{n}}
\sum_{k_{1n}=0}^{\infty}\binom{\Gamma_{1n}}{k_{1n}}
\varphi(a^1_{i_1+\sum_{\wm=2}^{n-1}\Gamma_{1\wm}-\wl+k_{1n}}w,v|
y_{n-1},y_{n})\nonumber\\
&\quad{}\times (y_{n-1}-y_{n})^{\Gamma_{1n}-k_{1n}+d}\nonumber\\
&=\Res_{y_{n-1}-y_{n}}
\sum_{k_{1n}=0}^{\infty}\binom{\Gamma_{1n}}{k_{1n}}\varphi(a^1_{i_1+\sum_{\wm=2}^{n-1}\Gamma_{1\wm}-\wl+k_{1n}}w\otimes v)\nonumber\\
&\quad{}\times 
(y_{n-1}-y_{n})^{-\wt a^1-\deg w+i_1+\sum_{\wm=2}^{n-1}\Gamma_{1\wm}-\wl+1+\Gamma_{1n}+d}\nonumber\\
&=\Res_{y_{n-1}-y_{n}}
\varphi(\Res_{x}(1+x)^{\Gamma_{1n}}x^{i_1+\sum_{\wm=2}^{n-1}\Gamma_{1\wm}-\wl}
Y_{W_1}(a^1,x)w\otimes v)\nonumber\\
&\quad{}\times 
(y_{n-1}-y_{n})^{-\wt a^1-\deg w+i_1+\sum_{\wn=2}^{n-1}\Gamma_{1\wm}-k+1+\Gamma_{1n}+d}.
\end{align}
By \eqref{eq:leq-2}, \eqref{eq:Gamma-a}, \eqref{eq:res-non-0}, and \eqref{eq:deform-a-2}, we have
\begin{align}
i_1\geq -\sum_{\wm=2}^{n-1}\Gamma_{1\wm}-1.
\end{align}
Since $i_1$ is the smallest element of $\{i_1,\ldots,i_{n-2}\}$, we have
\begin{align}
i_1+\cdots+i_{n-2}+i_n\geq -(n-2)(\sum_{\wm=2}^{n-1}\Gamma_{1\wm}+1)+i_{n}
\end{align}
and therefore
\begin{align}
\label{eq:i-n-bound}
i_n\leq -\Gamma-n+1+(n-2)(\sum_{\wm=2}^{n-1}\Gamma_{1\wm}+1)
\end{align}
by
\eqref{eq:total-degree-residue-g}. Thus by \eqref{eq:g-in}, we have
\begin{align}
\tf &\in \lws[[\wy_i-\wy_j\ |\ 1\leq i<j\leq n]][(y_1-y_{n})^{-1},\ldots,(y_{n-1}-y_{n})^{- 1}].
\end{align}
Since $g$ is homogeneous of total degree $\Gamma$, we have
\begin{align}
\tf &\in \lws[(y_1-y_{n})^{\pm 1},\ldots,(y_{n-1}-y_{n})^{\pm 1}]
\end{align}
and therefore there exists 
\begin{align}
\varphi(a^1,\ldots,a^{n-2},u,v| y_1,\ldots,y_{n})
\in \lws[(\wy_i-\wy_j)^{\pm 1}\ |\ 1\leq i<j\leq n]
\end{align}
such that
\begin{align}
\label{eq:pr-varphi}
&\iota_{(n-1,n)}\varphi(a^1,\ldots,a^{n-2},u,v| y_1,\ldots,y_{n})\nonumber\\
&=\varphi(\hY_{W_1}(a^1,\ldots,a^{n-2},u | \wy_1,\ldots,\wy_{n-2},\wy_{n-1}),v | \wy_{n-1},\wy_{n})
\end{align}
by the definition \eqref{eq:definition-g} of $\tf$.
If one put 
\begin{align}
&f_n(a^1,\ldots,a^{n-2},u,v|y_1,\ldots,y_n)\nonumber\\
&=\varphi(a^1,\ldots,a^{n-2},u,v| y_1,\ldots,y_{n}),\ n=2,3,\ldots
\end{align}
for $a^1,\ldots,a^{n-2}\in V, u\in W_1$ and $v\in \lwt$,
then $f_n\ (n=2,3,\ldots)$ satisfy \eqref{eq:permutation}--\eqref{eq:a-v} by \eqref{eq:pr-varphi}.
It follows from  Lemma \ref{lemma:rational-function} that
there exists a map $\Phi : S(W_1,\lwt)\rightarrow \lws$
such that
\begin{align}
&\Phi(Y_{S(W_1,\lwt)}(a^1,\wy_1-\wy_{n})\cdots,
Y_{S(W_1,\lwt)}(a^s,\wy_{s}-\wy_{n})\nonumber\\
&\qquad{}\times Y_{S(W_1,\lwt)}(u,\wy_{n-1}-\wy_{n})\nonumber\\
&\qquad{}\times Y_{S(W_1,\lwt)}(a^{s+1},\wy_{s+1}-\wy_{n})\cdots Y_{S(W_1,\lwt)}(a^{n-2},\wy_{n-2}-\wy_{n})
v)\nonumber\\
&=\iota_{(1,n),\ldots,(s,n),(n-1,n),(s+1,n),\ldots,(n-2,n)}
\varphi(a^1,\ldots,a^{n-2},u,v| y_1,\ldots,y_{n})
\end{align}
for $a^1,\ldots,a^{n-2}\in V, u\in W_1$, $v\in \lwt$,  and $s=1,\ldots,n-2$.

We define an $A(V)$-module homomorphism $\mu^{\prime} : \lws^{*}\rightarrow S(W_1,\lwt)^{\prime}(0)$ by
\begin{align}
\langle \mu^{\prime}(w_{(3)}^{\prime}),u\rangle
&=\langle w_{(3)}^{\prime},\Phi(u)\rangle
\end{align}
for $w_{(3)}^{\prime}\in \Omega_3^{*}$ and $u\in S(W_1,\lwt)(0)$.
By the universality of the generalized Verma module 
$S(\lws^{*})$, we have a $V$-module homomorphism 
\begin{align}
\mu^{\prime} : S(\lws^{*})\rightarrow S(W_1,\lwt)^{\prime}
\end{align}
and therefore its dual
\begin{align}
\mu^{\prime\prime} : S(W_1,\lwt)^{\prime\prime}\rightarrow S(\lws^{*})^{\prime}.
\end{align}
Restricting $\mu^{\prime\prime}$ to $S(W_1,\lwt)$, we have
\begin{align}
\mu : S(W_1,\lwt)\rightarrow S(\lws^{*})^{\prime}.
\end{align}
Then the map 
\begin{align}
W_1\otimes_{\C} S(\Omega_2)&\rightarrow S(\lws^{*})^{\prime}\db{x}\nonumber\\
u\otimes v&\mapsto \sum_{i\in\Z}\mu(u(i)v)x^{-i-1}
\end{align}
is the desired $\Z$-graded intertwining operator by the definition of $\mu$.
\end{proof}

The following result is a direct consequence of Theorem \ref{theorem:correspondence}.
\begin{corollary}\label{corollary:correspondence}
	Let $W_i=\oplus_{j=0}^{\infty}W_i(j), i=1,2,3$ be three $\N$-graded weak $V$-modules such that
	$W_2$ and $W_3^{\prime}$ are generalized Verma $V$-modules 
	and $\dim_{\C}W_3(j)<\infty$ for all $j\in\N$.
	Then, the  map
	\begin{align}
	\label{eq:one-one-iinter-N-c}
	\inter\binom{W_3}{W_1\ W_2}
	&\rightarrow \Hom_{A(V)}(A(W_1)\otimes_{A(V)}W_2(0),W_3(0))\nonumber\\
	\Phi(\ ,x)&\mapsto o^{\Phi}
	\end{align}
	is a linear isomorphism.
\end{corollary}

\begin{remark}
For arbitrary $c,h\in \C$, $M_{c,h}$ denotes the Verma module 
for the Virasoro algebra of central charge $c$ with lowest weight $h$
and $M_{c}$ denotes the quotient space of $M(c,0)$ by the submodule generated by $L(-1){\bf 1}$
where {\bf 1} is a lowest weight vector of weight $0$ for $M(c,0)$.
Then, $M_{c}$ is a vertex operator algebra and $M_{c,h}$ is an $M_{c}$-module (cf. \cite{LL}).
If $V=M_c$ and $W_1=M_{c,h}$ in  Theorem \ref{theorem:correspondence}, then
using the same computation as in \cite[Section 2]{Li}, we can describe 
the right-hand side of \eqref{eq:one-one-iinter-N} as follows. 
It is shown in \cite[Section 4]{FZ} (see also \cite[Proposition 3.1]{DMZ}) that
\begin{align}
A(M_c)&\rightarrow \C[t]\nonumber\\
(L(-2)+L(-1))^n&\mapsto t^n,\quad n\in\N
\end{align}
and
\begin{align}
A(M_{c,h})&\rightarrow\C[t_1,t_2]\nonumber\\
(L(-2)+2L(-1)+L(0))^m((L(-2)+L(-1))^nv_{h}&\mapsto t_1^{m}t_2^n,\quad m,n\in\N
\end{align}
are isomorphisms where $v_h$ is a non-zero element of $M(c,h)$ with weight $h$ 
and
the $\C[t]$-bimodule structure on $\C[t_1,t_2]$
is given by
\begin{align}
t^n\cdot f(t_1,t_2)&=t_1^nf(t_1,t_2)\mbox{ and }\nonumber\\
f(t_1,t_2)\cdot t^n&=t_2^nf(t_1,t_2)
\end{align}  
for $n\in\N$ and $f(t_1,t_2)\in \C[t_1,t_2]$. Thus, we have
\begin{align}
&\Hom_{A(M_c)}(A(M_{c,h})\otimes_{A(M_c)}\lwt,\lws)\nonumber\\
&\cong
\Hom_{\C[t]}(\C[t_1,t_2]\otimes_{\C[t]}\lwt,\lws)\nonumber\\
&\cong
\Hom_{\C}(\lwt,\lws).
\end{align}
\end{remark}

\section{Notation}\label{section:notation}
\begin{tabular}{ll}
$M[\wx,\wx^{-1}]_{[p,q]}$&$=\{\sum_{i=p}^{q}u_{i}x^{i}\ |\ u_p,u_{p+1},\ldots,u_{q}\in M\}$.\\
$A(V)$ & the Zhu algebra of a vertex operator algebra $V$.\\
$A(M)$ & the $A(V)$-bimodule associated with a weak $V$-module $M$.\\
$\Res_{x}f(x)$ & $=f_{-1}$ for $f(x)=\sum_{i\in\Z}f_ix^i$.\\
$I\binom{W_3}{W_1\ W_2}$&
the space of all intertwining operators of type $\binom{W_3}{W_1\ W_2}$.\\
$I_{\log}\binom{W_3}{W_1\ W_2}$&
the space of all logarithmic intertwining operators of type $\binom{W_3}{W_1\ W_2}$.\\
$I_{\Z}\binom{W_3}{W_1\ W_2}$&
the space of all $\Z$-graded intertwining operators of type $\binom{W_3}{W_1\ W_2}$.\\
$\Omega_{(2)},\Omega_{(3)}$ & left $A(V)$-modules.\\
$\Delta$ & $V=\bigoplus_{i=\Delta}^{\infty}V_i$.\\
$T(U)$ & the tensor algebra of a vector space $U$.\\
$T(W_1,\lwt)$ & $=T((V\oplus W_1\oplus \lwt)[t,t^{-1}])$.\\
$\free(W_1,\lwt)$ & $=T(V[t,t^{-1}])\otimes_{\C} W_1[t,t^{-1}]\otimes_{\C}T(V[t,t^{-1}])\otimes_{\C}\lwt$.\\
$S(U)=\oplus_{j=0}^{\infty}S(U)(j)$ & the generalized Verma module with $S(U)(0)=U$\\
& where $U$ is a left $A(V)$-module.\\
$J(W_1,\lwt)$ & Definition \ref{definition:ideal}.\\
$S(W_1,\lwt)$ & $=\free(W_1,\lwt)/\ideal (W_1,\lwt)$, \eqref{eq:definition-S}.\\
${U}_{\{\wy_1,\ldots,\wy_n\}}$&
$=U[[\wy_i-\wy_j\ |\ 1\leq i<j\leq n]][(\wy_i-\wy_j)^{-1}\ |\ 1\leq i<j\leq n],
\eqref{eq:Uwx-wx}$.\\
$\iota_{(i,j)}$ & a linear map ${U}_{\{\wy_1,\ldots,\wy_n\}}\rightarrow 
{U}_{\{\wy_1,\ldots,\widehat{\wy_{i}},\ldots,\wy_n\}}\db{\wy_{i}-\wy_{j}}$ defined by \eqref{eq:iotaij}.\\
$\iota_{(i_1,j_1),\ldots,(i_k,j_k)}$ & \eqref{eqn:iota-i-2-j-2} and \eqref{eqn:iota-i-k-j-k}.\\
$W^{*}$ & $=\Hom_{\C}(W,\C)$.\\
$W^{\prime}$ & $=\oplus_{i=0}^{\infty}\Hom_{\C}(W(i),\C)$ for $W=\oplus_{i=0}^{\infty}W(i)$.
\end{tabular}

\section*{Acknowledgements}

I thank Haisheng Li for useful comments.
I thank the anonymous referee for helpful comments.

\end{document}